\documentclass{article}
\usepackage{latexsym}
\usepackage{float}
\usepackage{pstricks}
\usepackage{amssymb}
\usepackage{amsmath}
\usepackage{amsfonts}
\usepackage{mathrsfs}
\usepackage{mathtools}

\mathtoolsset{showonlyrefs}
\usepackage{amsthm,array}
\usepackage{array}
\usepackage{cite}
\usepackage{graphicx}

\def\R{{\mathbb{R}}}

\def\a{\alpha}

\def\id{{\text{\rm Id}}}

\def\noi{\noindent}

\def\sign{\text{\rm sign\,}}

\def\sdeg{S^1\text{-Deg\,}}

\def\s1deg{S^1\text{\rm -Deg\,}}

\newcommand\cD{\ensuremath{\mathcal D}}

\newcommand\cF{\ensuremath{\mathcal F}}

\newcommand\cO{\ensuremath{\mathcal O}}
\newcommand\cP{\ensuremath{\mathcal P}}

\newcommand\cS{\ensuremath{\mathcal S}}
\newcommand\cT{\ensuremath{\mathcal T}}

\newcommand\fA{\ensuremath{\mathscr A }}

\newcommand\fF{\ensuremath{\mathfrak F}}

\newcommand\fQ{\ensuremath{\mathfrak Q}}
\newcommand\fR{\ensuremath{\mathfrak R}}
\newcommand\fS{\ensuremath{\mathfrak S}}

\newcommand\bbC{\ensuremath{\mathbb C}}

\newcommand\bbN{\ensuremath{\mathbb N}}

\newcommand\bbR{\ensuremath{\mathbb R}}

\newcommand\bbZ{\ensuremath{\mathbb Z}}

\newrgbcolor{violet}{.6 .1 .8}
\newrgbcolor{lightyellow}{1 1 .8}
\newrgbcolor{lightblue}{.80 1 1}
\newrgbcolor{mygreen}{0 .66 .05}
\definecolor{mygreen}{rgb}{0,.66,.05}
\definecolor{lightyellow}{rgb}{1,1,.80}
\newrgbcolor{orange}{1 .6 0}
\newrgbcolor{GreenYellow}{.85 1 .31}
\newrgbcolor{Yellow}{1  1  0}
\newrgbcolor{Goldenrod}{1  .90  .16}
\newrgbcolor{Dandelion}{1  .71  .16}
\newrgbcolor{Apricot}{1  .68  .48}
\newrgbcolor{Peach}{1  .50  .30}
\newrgbcolor{Melon}{1  .54  .50}
\newrgbcolor{YellowOrange}{1  .58  0}
\newrgbcolor{Orange}{1  .39  .13}
\newrgbcolor{BurntOrange}{1  .49  0}
\newrgbcolor{Bittersweet}{1.  .4300  .24}
\newrgbcolor{RedOrange}{1  .23  .13}
\newrgbcolor{Mahogany}{1.  .4475  .4345}
\newrgbcolor{Maroon}{1.  .4084  .5376}
\newrgbcolor{BrickRed}{1.  .3592  .3232}
\newrgbcolor{Red}{1  0  0}
\newrgbcolor{OrangeRed}{1  0  .50}
\newrgbcolor{RubineRed}{1  0  .87}
\newrgbcolor{WildStrawberry}{1  .04  .61}
\newrgbcolor{Salmon}{1  .47  .62}
\newrgbcolor{CarnationPink}{1  .37  1}
\newrgbcolor{Magenta}{1  0  1}
\newrgbcolor{VioletRed}{1  .19  1}
\newrgbcolor{Rhodamine}{1  .18  1}
\newrgbcolor{Mulberry}{.6668  .1180  1.}
\newrgbcolor{RedViolet}{.9538  .4060  1.}
\newrgbcolor{Fuchsia}{.5676  .1628  1.}
\newrgbcolor{Lavender}{1  .52  1}
\newrgbcolor{Thistle}{.88  .41  1}
\newrgbcolor{Orchid}{.68  .36  1}
\newrgbcolor{DarkOrchid}{.60  .20  .80}
\newrgbcolor{Purple}{.55  .14  1}
\newrgbcolor{Plum}{.50  0  1}
\newrgbcolor{Violet}{.98 .15 .95}
\newrgbcolor{RoyalPurple}{.25  .10  1}
\newrgbcolor{BlueViolet}{.84  .38  .98}
\newrgbcolor{Periwinkle}{.43  .45  1}
\newrgbcolor{CadetBlue}{.38  .43  .77}
\newrgbcolor{CornflowerBlue}{.35  .87  1}
\newrgbcolor{MidnightBlue}{.4414  .9259  1.}
\newrgbcolor{NavyBlue}{.06  .46  1}
\newrgbcolor{RoyalBlue}{0  .50  1}
\newrgbcolor{Blue}{0  0  1}
\newrgbcolor{Cerulean}{.06  .89  1}
\newrgbcolor{Cyan}{0  1  1}
\newrgbcolor{ProcessBlue}{.04  1  1}
\newrgbcolor{SkyBlue}{.38  1  .88}
\newrgbcolor{Turquoise}{.15  1  .80}
\newrgbcolor{TealBlue}{.1572  1.  .6668}
\newrgbcolor{Aquamarine}{.18  1  .70}
\newrgbcolor{BlueGreen}{.15  1  .67}
\newrgbcolor{Emerald}{0  1  .50}
\newrgbcolor{JungleGreen}{.01  1  .48}
\newrgbcolor{SeaGreen}{.31  1  .50}
\newrgbcolor{Green}{0  1  0}
\newrgbcolor{ForestGreen}{.1992  1.  .2256}
\newrgbcolor{PineGreen}{.3100  1.  .5575}
\newrgbcolor{LimeGreen}{.50  1  0}
\newrgbcolor{YellowGreen}{.56  1  .26}
\newrgbcolor{SpringGreen}{.74  1  .24}
\newrgbcolor{OliveGreen}{.6160  1.  .4300}
\newrgbcolor{RawSienna}{.53  .28  .16}
\newrgbcolor{Sepia}{1.  .7510  .70}
\newrgbcolor{Brown}{.41  .25  .18}
\newrgbcolor{TAN}{.86  .58  .44}
\newrgbcolor{Gray}{1.  1.  1.}
\newrgbcolor{Black}{1  1  1}
\newrgbcolor{White}{1  1  1}

\newtheorem{theorem}{Theorem}[section]
\newtheorem{proposition}[theorem]{Proposition}
\newtheorem{lemma}[theorem]{Lemma}

\newtheorem{definition}[theorem]{Definition}
\newtheorem{remark}[theorem]{Remark}
\newtheorem{example}[theorem]{Example}

\newtheorem{remark-definition}[theorem]{Remark and Definition}

\begin{document}
\title{Sliding Hopf bifurcation in interval systems}

\author{E. Hooton$^{1,*}$, Z. Balanov$^{1}$,  W. Krawcewicz$^{1}$, D. Rachinskii$^{1}$}

\date{}




\maketitle

\footnote{$^{1}$ Department of Mathematical Sciences, University of Texas at Dallas, Richardson, Texas, 75080 US.
Emails: exh121730@utdallas.edu, balanov@utdallas.edu, wieslaw@utdallas.edu, dmitry.rachinskiy@utdallas.edu}

\footnote{$^{*}$ Corresponding author. Mailing address: Department of Mathematical Sciences, The University of Texas at Dallas, 800 W. Campbell Rd, Richardson, Texas, 75080 US. Email: exh121730@utdallas.edu}

\begin{abstract}
In this paper, the equivariant degree theory is used to analyze the occurrence of the Hopf bifurcation under effectively verifiable mild conditions. We combine the abstract result
with standard interval polynomial techniques based on Kharitonov's theorem to show the existence of a branch of periodic solutions emanating from the equilibrium in the settings relevant to robust control. The results are illustrated with a number of examples.
\end{abstract}

\section{Introduction}\label{sec:Intro}

{\bf Subject and goal.}
Many problems in
population dynamics, neural networks, fluid dynamics, solid mechanics, elasticity,
chemistry, mechanical and electrical engineering lead to studying the so-called Hopf bifurcation (more precisely, Poincar\'e-Andronov-Hopf bifurcation) in
dynamical systems parameterized by a real parameter (see, for example, \cite{MarsdenMcK,AED,GolSchSt,IV-B,KW} and references therein). To be more specific,
given a parameterized family
\begin{equation}\label{eq:sys-HB}
 \dot{x} = f(\alpha,x), 
 \qquad \alpha \in [\alpha_-,\alpha_+], \; x \in \mathbb R^d,
\end{equation}
where $f : [\alpha_-,\alpha_+] \times \mathbb R^d \to \R^d$ is a continuous map and $(\alpha,0)$ is a curve of trivial stationary solutions,
the Hopf bifurcation is a phenomenon occurring when $\alpha$ crosses
some critical value $\alpha_o$ (for which the linearization $D_x f(\alpha,0)$ admits a purely imaginary eigenvalue)
and resulting in appearance of a branch of small amplitude periodic solutions near the curve $(\alpha, 0)$. In his original work \cite{Hopf}, E.~Hopf studied system \eqref{eq:sys-HB} under the following assumptions:  (a) $f$ is analytic in both variables; (b) for $\alpha = \alpha_o$, exactly two complex conjugate characteristic roots $\mu(\alpha)$ and  $\overline{\mu(\alpha)}$ intersect the imaginary axis (absence of multiple/resonant roots); (c) $\mu(0) \not=0$ (exclusion of steady-state bifurcation); and, (d) ${\rm Re}\, \mu^{\prime}(0) \not=0$ (transversality).
Hopf's theorem includes conditions for the occurrence of the bifurcation (i.e., the existence result) and conditions for stability of small cycles bifurcating from the stationary point.
After this pioneering work, a substantial effort was made in order to relax conditions (a)--(d) (see, for example, \cite{MarsdenMcK,AED,GolSchSt,IV-B,KW,ChowMPYorke,KK} and references therein).  One objective of this paper is  to present an abstract result on the occurrence of the Hopf bifurcation in  \eqref{eq:sys-HB} under very mild (and effectively verifiable) hypotheses 
containing many known occurrence results as a particular case (cf.~Theorem \ref{thm:main-theorem}, Theorem \ref{thm:corollaries-all} and Remark \ref{rem:comparison}).
It should be stressed that we do not study stability of bifurcating periodic solutions.

Our choice of the conditions on the nonlinearity $f$ and its derivative $D_xf(\alpha,0)$ is essentially determined by the following observations. In analysis and design, it is customary to deal with approximations of complex  models that have some degree of uncertainty (one can think of the so-called nominal systems widely used in robust control; see, for example, \cite{BCK}).
Considering a model with  uncertain parameters, one can expect that the entries of the matrix  $D_xf(\alpha,0)$ belong to some known intervals of values rather than
being represented by fixed numbers. This suggests to study the Hopf bifurcation phenomenon for a class of systems \eqref{eq:sys-HB}
 where coefficients of the linearization are limited to known intervals. In this setting, the characteristic polynomial of $D_xf(\alpha,0)$ that defines the stability properties of the linearization
also becomes an {\it interval polynomial} (see, for example, \cite{BCK}).
Importantly, this setting includes the scenario when the characteristic values of the linearization of a representative system \eqref{eq:sys-HB} {\it slide} along the imaginary axis when the bifurcation parameter is varied (see Figure \ref{fig:Sliding-Introduction}a).
\begin{figure}[ht]
\begin{center}
\includegraphics*[width=0.3\columnwidth]{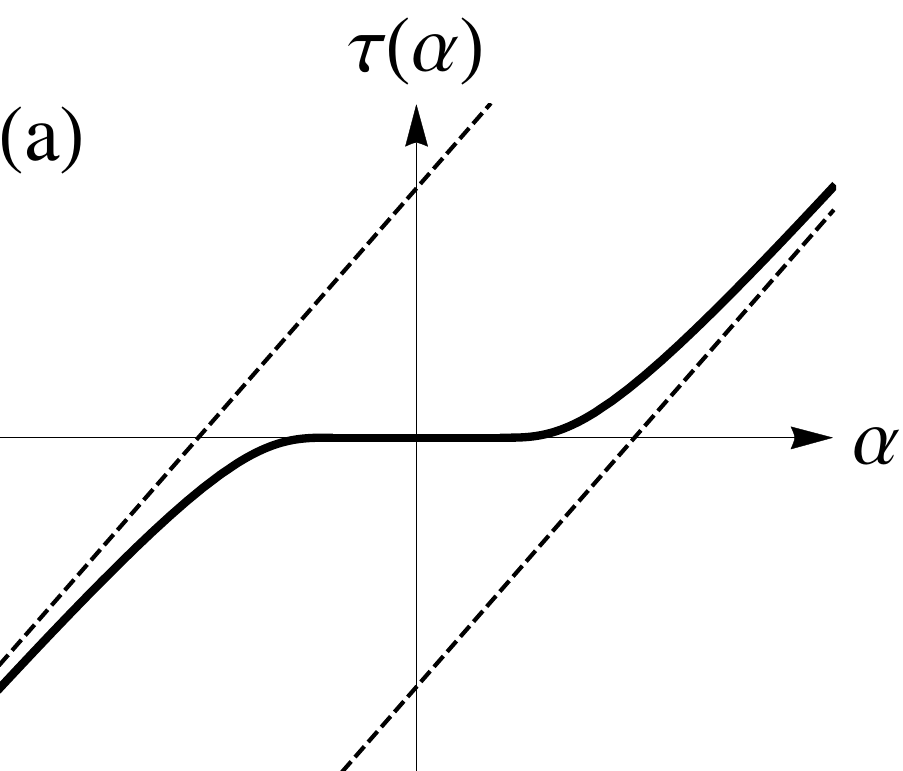} \qquad\quad \includegraphics*[width=0.3\columnwidth]{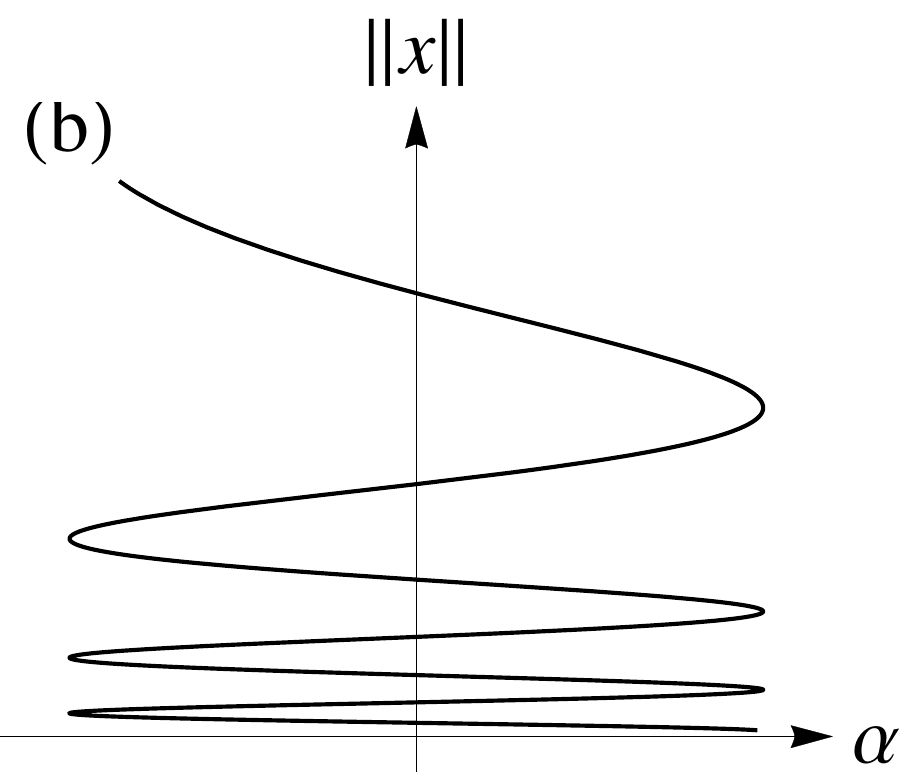}
\end{center}
\caption{(a) Given an $\alpha$-parametrized family of characteristic polynomials with unknown coefficients that are limited to some intervals, the dashed lines bound a corridor
for the real part $\tau(\alpha)={\rm Re}\, \mu(\alpha)$ of an eigenvalue, while the solid line indicates a sliding scenario for some selector of the family.
(b) Possible complex behavior of the branch of periodic solutions.}
\label{fig:Sliding-Introduction}
\end{figure}
The main goal of the present paper is to propose a method for analysis of the occurrence of the Hopf bifurcation in the presence of such sliding.

As a matter of fact, the sliding phenomenon makes the problem non-local. Namely, it does not allow one to localize a bifurcation point
on the basis of the knowledge of the linearization, that is based on the condition ${\rm Re}\, \mu(\alpha)=0$ (see Figure \ref{fig:Sliding-Introduction} a,b).


To study the Hopf bifurcation in this setting, one needs to deal with the whole interval of sliding that consists of potential bifurcation points.
Thus, sliding is in sharp contrast to the transversality condition (d) above. At the same time, to the best of our knowledge, all the existing results on the occurrence of Hopf bifurcation identify explicitly a critical value of the parameter $\alpha$ at which {\rm Re}\,$\mu(\alpha)$ changes its sign (for the least restrictive condition of this type, we refer to \cite{KK}). Some conditions for the existence of a branch of cycles that are non-local with respect to the parameter can be found in \cite{KRB, KRB1, KRB2}.

The simplest scenario which includes sliding and is covered by our results is the following.
Suppose that system  \eqref{eq:sys-HB} has an equilibrium $x=0$ for all values of the parameter $\alpha\in [\alpha_-,\alpha_+]$.
Assume that the linearization $D_x f(\alpha,0)$ of the right hand side is invertible
and has {\em at most} one pair of purely imaginary eigenvalues for any $\alpha\in [\alpha_-,\alpha_+]$. Finally, assume that the zero equilibrium
is hyperbolic for $\alpha=\alpha_\pm$ and the dimension of the stable manifold of the linearization of \eqref{eq:sys-HB} at zero
is different for $\alpha=\alpha_-$ and $\alpha=\alpha_+$. Then there is a Hopf bifurcation point on the interval $(\alpha_-,\alpha_+)$.
Theorem \ref{thm:main-theorem} presented below also covers  more complex scenarios including multiple and resonant eigenvalues of the linearization on the imaginary axis.


{\bf Method.}
In \cite{Hopf}, the Hopf bifurcation in \eqref{eq:sys-HB} was studied based on the series expansion of $f$. The further progress was related to the methods rooted in the singularity theory: assuming that the system satisfies several regularity and genericity conditions, one can combine the normal form classification
with Center Manifold Theorem/averaging method/Lyapunov-Schmidt reduction. For a detailed exposition of these
concepts and related techniques, we refer to \cite{GS,GolSchSt,MarsdenMcK}.

Being very effective in the settings they are usually applied to, the singularity theory based methods meet difficulties if a setting is not regular/generic enough. For example,
dynamical systems with hysteresis components admit linearization at the origin while any small neighborhood of the origin contains non-differentiability points which makes
the Center Manifold Reduction impossible (see \cite{BKRZ,d1,d2,d3,d4,d5,d6}) for details).  As long as the stability of bifurcating solutions is not questioned, one can use homotopy theory based methods. Important steps in this direction were done in \cite{AY} (framed bordism theory), \cite{ChowMPYorke} (Fuller index), \cite{KK} (parameter functionalization method combined with the Leray-Schauder degree), to mention a few.

During the last twenty years the {\it equivariant degree theory} emerged in non-linear analysis (for the detailed
exposition of this theory, including historical remarks, we refer to
recent monographs \cite{AED,IV-B} and surveys \cite{BKRHandbook,survey,Ize-Handbook}; for
the prototypal invariants, see \cite{Dancer, Dancer1, Fuller,KB}).
The equivariant degree, being the main topological tool used in this paper, is an instrument that allows
``counting'' orbits of solutions to symmetric equations in the same
way as the usual Brouwer degree does, but according to their
symmetry properties. In particular, the equivariant degree theory
has all the attributes allowing its application in
non-smooth and non-generic equivariant settings related to equivariant dynamical systems
having, in general, infinite dimensional phase spaces with lack of linear structure (cf.~\cite{BKRZ}). We refer to \cite{IV-B,AED} and references therein for the equivariant degree treatment of the (symmetric) Hopf bifurcation in different environments (see also \cite{Kiel}).  In the present paper, we use the $S^1$-degree with one free parameter (see \cite{AED}
for the axiomatic approach).

Theorem \ref{theor:Haritonov2} below explicitly refers to the verification of stability properties of interval polynomials (cf.~conditions
{\bf {(R3)}} and {\bf {(R4)}}). Among very few results on the connection between
perturbations of the coefficient and root locations, Kharitonov's theorem (\cite{Kharitonov}, see also \cite{BCK,HK}) takes a firm position. To be more specific,
V. L. Kharitonov showed that given a family of interval polynomials with real coefficients,
it is necessary and sufficient to test just four canonically defined  members of the family
in order to decide that all polynomials are Hurwitz stable. The main topological ingredient of Kharitonov's proof is the so-called
Zero Exclusion Principle (in short ZEP) which can be traced back to the classical Argument principle in Complex Analysis. In this paper, combining  ZEP with simple combinatorial
arguments, we establish a Kharitonov type result for the so-called $k$-stable interval polynomials (cf.~Lemma \ref{Interval_one_root} and Definition \ref{def-monic-k-stable}). In particular, it shows that Kharitonov's approach is sensitive not only to Hurwitz stability, but also to the change of the dimension of the stable manifold in families of interval polynomials which is crucial for studying the Hopf bifurcation phenomenon.


The paper is organized as follows. In the next section, 
we present some background related to the
Hopf bifurcation and interval polynomials. In Section \ref{sec:formulations},  main results are formulated
(see Theorems  \ref{thm:main-theorem}, \ref{thm:corollaries-all} and \ref{theor:Haritonov2}).
Some examples illustrating Theorems \ref{thm:corollaries-all} and \ref{theor:Haritonov2} are given in Section \ref{sec:examples}.
Section \ref{sec:main-theorem} contains the proof of Theorem \ref{thm:main-theorem} which
is close in spirit to the proofs of Theorems 9.18 and 9.24 from \cite{AED}. In Section \ref{sec:proofs-remaining}, we provide the proofs of remaining results.
A brief summary of properties of the $S^1$-equivariant degree is presented in Appendix.

\section{Preliminaries}\label{sec:prelim}

\subsection{Hopf bifurcation}
The Hopf bfurcation being the main subject of the present paper is formalized in the following definition (cf.~\cite{KK,AED}).

\begin{definition}\label{YTM}{
Consider a non-empty set $\Gamma$ of non-constant periodic solutions $(\a,p,x(t))$ of system (\ref{eq:sys-HB}) (where $p$ is the minimal period of $x(t)$) such that 
$p\in [p_-,p_+]\subset (0,\infty)$. 
The set $\Gamma$ is called a branch bifurcating  from the trivial solution if the union of $\Gamma$ and the set of trivial solutions, $\Gamma\bigcup\, [\alpha_-,\alpha_+]\times [p_-,p_+]\times \{x=0\}$, is a connected compact set.}
\end{definition}

If $\Gamma$ is a branch of non-constant periodic solutions bifurcating  from the trivial solution, then the interval $[\alpha_-,\alpha_+]$ contains at least one Hopf bifurcation point $\alpha_0$
in the weak sense of \cite{KK}. In other words, there are converging sequences $\alpha_k\to \alpha_0$ and $p_k\to p_0>0$ such that system (\ref{eq:sys-HB}) with $\alpha=\alpha_k$ has a non-constant
periodic solution $x_k(t)$ with the minimal period $p_k$ and $\|x_k\|_C\to 0$. If the necessary condition for the Hopf bifurcation (see, Section \ref{ddd}) is satisfied
at exactly one point $\alpha_0\in (\alpha_-,\alpha_+)$, then Definition \ref{YTM} reduces to the definition of the Hopf bifurcation used in \cite[p.~260]{AED}.  However,
the setting of Definition \ref{YTM} does not exclude a possibility of more complex behavior of the branch shown in Figure \ref{fig:Sliding-Introduction}b in the case of an eigenvalue
sliding along the imaginary axis as in Figure \ref{fig:Sliding-Introduction}a.

\subsection{Interval polynomials and Kharitonov's theorem}
Following \cite{BCK}, an interval matrix, denoted
$$\fA = (I_{kj})^n_{k,j=0}  =\{A:A_{kj}\in I_{kj},\ k,j=1,\ldots,n\},$$
is the set of all matrices whose $(k,j)$-th entry lies in the interval $I_{kj}$. Similarly, for interval polynomials,
$$
\cS = I_0 + I_1\lambda+\cdots+I_{n-1}\lambda^{n-1}+I_n\lambda^n =\{P=a_0+a_1\lambda+\cdots+a_n\lambda^n: a_k\in I_k,\ k=1,\ldots,n\}.
$$
Naturally, ${\cal S}_1+ {\cal S}_2=\{P_1+P_2:\ P_1\in {\cal S}_1,\ P_2\in {\cal S}_2\}$ and ${\cal S}_1\cdot {\cal S}_2=\{P_1P_2:\ P_1\in {\cal S}_1,\ P_2\in {\cal S}_2\}$.

Also, we will need the following definition.

\begin{definition}\label{def-monic-k-stable}

(i) A polynomial  with real coefficients is called {\it monic} if its highest coefficient is one. An interval polynomial $\cS$ is called {\it monic} if any $P \in \cS$ is monic.

\smallskip

(ii) Let $\cS$ be a monic interval polynomial of degree $n$. We say that $\cS$ is $q$-stable (resp., $q$-unstable) if for any $P\in\cS$, $P$ has exactly $q$ roots with ${\rm Re}\,(z)<0$ and $n-q$ roots with ${\rm Re}\,(z)>0$ (resp., $q$ roots with ${\rm Re}\,(z)>0$ and $n-q$ roots with ${\rm Re}\,(z)<0$).
\end{definition}

The classical Hurwitz stability is, therefore, called $0$-instability in our terminology.  Given an interval polynomial
\begin{equation}\label{interval_polynomial}
 \cS = I_0 + I_1\lambda +\cdots+I_{n-1}\lambda ^{n-1}+\lambda ^n,  \quad I_j=[a_j,b_j],
\end{equation}
we denote
\begin{equation}\label{eq:corner-polynomials}
\begin{aligned}
g_1(\cS, \lambda)&= a_0+b_2\lambda ^2+a_4\lambda ^4+\cdots;\qquad
&g_2(\cS, \lambda)&= b_0+a_2\lambda ^2+b_4\lambda ^4+\cdots; \\
h_1(\cS, \lambda) &= a_1\lambda +b_3\lambda ^3+a_5\lambda ^5+\cdots;\qquad
&h_2(\cS, \lambda)&= b_1\lambda +a_3\lambda ^3+b_5\lambda ^5+\cdots.
\end{aligned}
\end{equation}
Notice that for any $P\in\cS$,
\begin{equation}\label{rectangle_inequalities}
\begin{aligned}\text{Re}\,(g_1(\cS, i\omega))\leq \text{Re}\,(P(i\omega))\leq  \text{Re}\,(g_2(\cS, i\omega)),\quad
\text{Im}\,(h_1(\cS, i\omega))\leq\text{Im}\,(P(i\omega))\leq \text{Im}\,(h_2(\cS, i\omega)).
\end{aligned}
\end{equation}

The following classical result regarding stability of interval polynomials is known as Kharitonov's theorem (see \cite{Kharitonov,HK,BCK}).

\begin{theorem}[Kharitonov]\label{thm:Kharitonov}
The interval polynomial \eqref{interval_polynomial}
is Hurwitz stable if and only if the following polynomials are Hurwitz stable:
\begin{align*}
g_1(\cS,\cdot) & + h_1(\cS,\cdot), & g_1(\cS,\cdot) & + h_2(\cS,\cdot), &g_2(\cS,\cdot) & + h_1(\cS,\cdot), & g_2(\cS,\cdot) & + h_2(\cS,\cdot).
\end{align*}
\end{theorem}

We will use a $q$-unstable variant of Kharitonov's theorem.

\begin{lemma}\label{Interval_one_root}
If a polynomial $P_o\in \cS$ is $q$-unstable and $$({\rm Re}\,(g_1(\cS, i\omega))\cdot {\rm Re}\,(g_2(\cS, i\omega))\;,\; {\rm Im}\,(h_1(\cS, i\omega))\cdot {\rm Im}\,(h_2(\cS, i\omega))) \notin \{(x,y):x\leq0,y\leq0\}$$ for any $\omega\geq0$, then the interval polynomial $\cS$ is $q$-unstable.
\end{lemma}
\medskip

The main topological ingredient of the proof of both statements is the so-called

\medskip
\noindent
\textbf{Zero Exclusion Principle}. If some polynomial $P_o\in\cS$ is $q$-unstable and for any $P\in\cS$ and any $\omega > 0$, $P(i\omega)\neq 0$, then the interval polynomial $\cS$ is $q$-unstable.

\medskip
\noindent
{\bf Proof of Lemma \ref{Interval_one_root}:}
As an immediate consequence of inequalities \eqref{rectangle_inequalities} one has that,  for any $P\in\cS$,
\begin{equation}\label{rectangular-condition}
\begin{aligned}
\text{Re}\,(g_1(\cS, i\omega))\cdot \text{Re}(g_2(\cS, i\omega)) &> 0  \implies \text{Re}\,P(i\omega)\neq0, \\
\text{Im}\,(h_1(\cS, i\omega))\cdot \text{Im}(h_2(\cS, i\omega)) &> 0 \implies \text{Im}\,P(i\omega)\neq0.
\end{aligned}
\end{equation}
The result then follows from the Zero Exclusion Principle. \hfill $\Box$

\medskip

\subsection{Interval polynomials and Descartes' Criterion}

Recall the following classical result.

\medskip
\noindent
\textbf{Descartes' criterion}.
{\it If the terms of a single-variable polynomial with real coefficients are ordered by descending variable exponent, then the number of positive roots of the polynomial is less than or equal to the number of sign differences between consecutive nonzero coefficients.}

\medskip
\noindent
As an immediate consequence, we have


\begin{proposition}\label{prop:Decartes-imaginary}
Given a polynomial $P$ with real coefficients, assume that there exist polynomials $Q$ and $R$ such that the coefficients of the polynomial
\begin{equation}\label{eq:Decart-verif}
S(P,Q,R)(\omega) =  Q(\omega){\rm Re}(P(i\omega)) + R(\omega){\rm Im}(P(i\omega))
\end{equation}
have at most one sign change.
Then, $P$ may have at most one pair of purely imaginary roots.
\end{proposition}

\noindent
Indeed, for $\omega > 0$,  if $i\omega$ is a root of  $P$, then $\omega$ is a (positive) root of $S(P,Q,R)$.

\medskip
In what follows, we use an interval polynomial variant of Proposition \ref{prop:Decartes-imaginary}. For the precise formulation, we need the following definition.
Given an interval polynomial $\cS$,
we say that the coefficients of $\cS$ have at most one sign change if, for some $j$, either $[a_k,b_k]\subset (-\infty,0]$  for all $k < j$ and  $[a_k,b_k]\subset [0,\infty)$ for all $k > j$,
or $[a_k,b_k]\subset (-\infty,0]$  for all $k > j$ and  $[a_k,b_k]\subset [0,\infty)$ for all $k < j$.
Notice that if the coefficients of $\cS$ have at most one sign change then the coefficients of any polynomial $P \in \cS$ have  at most one sign change.

Set
$$
\cT(\cS,Q,R)(\omega)=Q(\omega){\rm Re}(\cS(i\omega)) + R(\omega){\rm Im}(\cS(i\omega)).
$$
\begin{lemma} \label{lem:interval-descartes} Assume that there exist $Q, R$ such that the coefficients of $\cT(\cS,Q,R)$ have at most one sign change. Then, any polynomial $P \in \cS$ has at most one pair of purely imaginary roots.
\end{lemma}

\begin{proof}
Suppose, for the contrary, that some $P\in \cS$ has more than one pair of purely imaginary roots. By \eqref{eq:Decart-verif},
$S(P,Q,R)(\omega) \in \cT(\cS,Q,R).$
Therefore, $S(P,Q,R)(\omega)$ has at least two distinct positive real roots. Hence, by Descartes' criterion, the coefficients of $S(P,Q,R)(\omega)$ have more than one sign change, which is a contradiction.
\end{proof}

\section{Main results}\label{sec:formulations}

\subsection{Abstract result}

Set  $V= \bbR^d$ and assume that $f: [\alpha_-,\alpha_+] \times V \to V$ is a map satisfying the following properties:

\medskip

\noi\textbf{(P0)} $f$ is continuous;

\medskip
\noi\textbf{(P1)} The Jacobi matrix $D_xf(\alpha,0)$ exists for all $\alpha$,  depends continuously on $\alpha$ and
\begin{equation}\label{eq:ref-uniform-converge}
\lim_{\|x\| \to 0} \sup_{\alpha}{\|f(\alpha,x) - D_xf(\alpha,0)x \|\over \|x\|} = 0;
\end{equation}

\medskip
\noi\textbf{(P2)} $f(\alpha,0) =0$ for all $\alpha$;

\medskip
\noi\textbf{(P3)} $\det(D_xf(\alpha,0))\neq 0$ for all $\alpha$.

\medskip
To formulate the next condition, take the map $\Lambda:[\alpha_-,\alpha_+] \times \bbR\times\bbR \to \bbC$ determined by the characteristic polynomial of the Jacobi matrix $D_xf(\alpha,0)$, i.e.
\begin{equation}\label{def:Lambda}
\Lambda (\alpha, \beta, \tau) = \text{det}_{\mathbb C}((\tau + i\beta)\text{Id} - D_xf(\alpha,0)).
\end{equation}
 Define a $\bbZ_2$-action on $\mathbb R^3$ by
$$
(\alpha,\beta,\tau)\mapsto (\alpha,-\beta,\tau).
$$
Also, given a set $\mathcal P \subset [\alpha_-, \alpha_+] \times \bbR \times \bbR_+$,
define
\begin{equation}\label{eq:P-0-set}
\mathcal{P}_{\pm} = \mathcal{P}\bigcap \left(\{\alpha_\pm\} \times \bbR \times \bbR_+\right) \quad {\rm and} \quad
\cP_0 = \mathcal{P}\bigcap \left([\alpha_-,\alpha_+] \times \bbR \times \{0\}\right),
\end{equation}
where $\bbR_+$ denotes the non-negative semi-axis. We will denote by $\partial \Omega$ the boundary of a domain ${\Omega}$
and by $\overline{\Omega}$ the closure of ${\Omega}$.

\medskip
\noi\textbf{(P4)} There exists a bounded $\bbZ_2$-invariant domain $\mathcal{P} \subset [\alpha_-, \alpha_+] \times \bbR \times \bbR_+$ such that:

(i) $\mathcal{P}$  is homeomorphic to a closed ball;

\smallskip

(ii) $\Lambda(\alpha,\beta, \tau) \neq 0$ for all $(\alpha,\beta, \tau) \in \partial\mathcal{P}\setminus (\mathcal{P}_+ \bigcup \mathcal{P}_- \bigcup \cP_0)$;

\smallskip

(iii) $\mathcal{P}_+$ and $\mathcal{P}_-$ contain a different number of roots of $\Lambda(\alpha,\beta,\tau)$ (counted according to their multiplicities).


\medskip
\noi\textbf{(P5)}
There exists a finite collection of disjoint sets $\cD_k \subset [\alpha_-,\alpha_+] \times \bbR_+$ such that:

(i) each $\cD_k$ is homeomorphic to a closed disk;

\smallskip

(ii) $\Lambda^{-1}(0)\bigcap (\bigcup \cD_k\times\{0\})= \Lambda^{-1}(0)\cap \cP_o$;

\smallskip

(iii) for any $l\in \bbN$ and for any $(\alpha,\beta)\in \partial\cD_k$, $\Lambda (\alpha, l \beta, 0) \neq 0$.

\medskip

\noindent

\begin{remark}\label{rem:crossing-numbers} {\rm Conditions \textbf{(P0)} and \textbf{(P1)} reflect the minimal regularity that
we require from system \eqref{eq:sys-HB}. Condition \textbf{(P2)} guarantees the existence of a branch of zero equilibria from which we expect the occurrence of the Hopf bifurcation, while \textbf{(P3)} excludes steady-state bifurcation.

The domain $\cP$ provided by \textbf{(P4)} acts as a ``trap'' catching the roots of $\Lambda$, which may potentially contribute to the Hopf bifurcation. Condition \textbf{(P4)}(ii) guarantees that the roots may only escape $\cP$ through the planes $\{\alpha = \alpha_-\},\; \{\alpha = \alpha_+\}$ and $\{\tau = 0\}$. Condition \textbf{(P4)}(iii) is an analog of the standard non-zero crossing number assumption.

 On the other hand, the sets $\cD_k$ provided by \textbf{(P5)} form the domain on which we will compute the topological invariant. Property \textbf{(P5)}(iii) (which is a kind of non-resonance condition) ensures that the topological invariant is well-defined, while \textbf{(P5)}(ii) (which says that all the roots in $\cD_k$ are precisely those ``exiting'' $\cP$) ensures that the invariant is non-trivial and thus that the Hopf bifurcation takes place.  Several versions of conditions \textbf{(P4)} and \textbf{(P5)} directly related to the classical setting for the Hopf bifurcation are discussed in the next subsection.

}
\end{remark}
The following statement is our main abstract result.
\begin{theorem}
\label{thm:main-theorem}
Let $f$ satisfy conditions \textbf{(P0) - (P5)}.
Then, there exists a branch of non-constant periodic solutions to system  \eqref{eq:sys-HB} bifurcating  from the trivial solution (cf.~Definition \ref{YTM}).
\end{theorem}

\subsection{Corollaries}
Let us consider some corollaries of Theorem \ref{thm:main-theorem} based on variations of conditions {\bf (P4)} and {\bf (P5)} which are more
relaxed but easier to verify. To this end, we introduce the following notation:
\begin{equation}\label{eq:def_R_S}
\begin{aligned}
R(f)&=\{(\alpha,\beta)\in [\alpha_-,\alpha_+]\times\bbR: \Lambda(\alpha,\beta,0)=0\},\\
S_j(f) &= \{(\alpha,\beta) \in [\alpha_-,\alpha_+]\times\bbR: (\alpha,j\beta)\in R(f)\}, \  j=2,3,\ldots,\\
S(f)&= \bigcup^\infty_{j=2} S_j(f).
\end{aligned}
\end{equation}
\begin{remark}{\rm
Notice that $R(f)$  is the set of purely imaginary characteristic roots lying between $\alpha_-$ and $\alpha_+$, while $S(f)$ is the set of points an integer multiple of which lies in $R(f)$.
}
\end{remark}

We use a few variants of conditions {\bf (P4)} and {\bf (P5)}.

\medskip

\noi\textbf{(P4$^\prime$)}
{\it There exist $\alpha_-,\alpha_+$, for which  $x=0$ is a  hyperbolic equilibrium of \eqref{eq:sys-HB} and the dimension of the unstable manifold of the linearization of \eqref{eq:sys-HB}  at $0$ is different for $\alpha_-$ and $\alpha_+$.}

\medskip

\noi\textbf{(P5$^\prime$)}  There exists a finite collection of disjoint sets $\cD_k \subset [\alpha_-,\alpha_+]\times \bbR_+$ such that:

(i) each $\cD_k$ is homeomorphic to a closed disk;

\smallskip

(ii) $R(f) \subset \bigcup \cD_k$;

\smallskip

(iii) $S(f) \bigcap \partial\cD_k = \emptyset$ for any $k$.

\medskip

\noi\textbf{(P5$^{\prime\prime}$)} $R(f)\cap S(f)=\emptyset$.

\medskip

\noi\textbf{(P5$^{\prime\prime\prime}$)} $D_xf(\alpha,0)$ has at most one pair of purely imaginary eigenvalues for all $\alpha \in [\alpha_-,\alpha_+]$.

\medskip

\noi\textbf{(P5$^{\prime\prime\prime\prime}$)} There exists a unique $\alpha \in (\alpha_-,\alpha_+)$ such that
$D_xf(\alpha,0)$ has purely imaginary eigenvalues.

\begin{remark}\label{rem:verification-conditions-P}
{\rm  Observe that \textbf{(P4$^\prime$)} is a non-zero crossing number condition; in particular, the classical Routh-Hurwitz criterion (see, for example, \cite{RahSchm}) can be useful for its verification. Condition \textbf{(P5$^{\prime}$)} is a slight modification of \textbf{(P5)}, adjusted to the case when \textbf{(P4$^\prime$)} holds. Condition \textbf{(P5$^{\prime\prime}$)} is the classical non-resonance condition. Condition \textbf{(P5$^{\prime\prime\prime}$)}, although much more restrictive than condition \textbf{(P5$^{\prime\prime}$)}, can be verified using Descartes' criterion (see also Proposition \ref{prop:Decartes-imaginary}). Finally,  \textbf{(P5$^{\prime\prime\prime\prime}$)} is the standard isolated center condition (see, for example, \cite{AED}). }
\end{remark}

The following statement is based on Theorem 3.2 and is used below to obtain sufficient conditions for the Hopf bifurcation in interval systems.

\begin{theorem}\label{thm:corollaries-all}
Suppose $f$ satisfies conditions \textbf{(P0) - (P3)}. Suppose, in addition, $f$ satisfies one of the following assumptions:

\smallskip
(a) \textbf{(P4$^\prime$)} and \textbf{(P5$^\prime$)};

\smallskip

(b) \textbf{(P4$^\prime$)} and \textbf{(P5$^{\prime\prime}$)};

\smallskip

(c) \textbf{(P4$^\prime$)} and \textbf{(P5$^{\prime\prime\prime}$)};

\smallskip
(d) \textbf{(P4)} and \textbf{(P5$^{\prime\prime}$)};

\smallskip
(e) \textbf{(P4)} and \textbf{(P5$^{\prime\prime\prime\prime}$)}.

\smallskip

\noindent
Then, system  \eqref{eq:sys-HB} has a branch of non-constant periodic solutions bifurcating  from the trivial one.
\end{theorem}

\begin{remark}\label{rem:comparison}
{\rm Under the assumption that $f$ is of class $C^{1,1}$, Theorem  \ref{thm:corollaries-all}(e) was established in \cite{Ize-1985} (see also \cite{Geba-Wacek,ChowMPYorke,AED,
AY}). On the other hand, by taking a sufficiently small neighborhood $(\alpha_-,\alpha_+)$, one can deduce the main result of \cite{KK} from
 Theorem \ref{thm:corollaries-all}(d) (without extra ``simplicity" assumptions on the corresponding eigenvalues).}
\end{remark}

\subsection{Theorem  \ref{thm:corollaries-all} and interval polynomials}
In this section, we address families of one-parameter systems for which every member is undergoing the Hopf bifurcation.
To be more precise, denote by  $\fA $ a map from $[\alpha_-,\alpha_+]$ to the set of interval matrices of size $d \times d$ and by $\mathfrak r$ a set of maps $r:[\alpha_-,\alpha_+]\times V \to V$.  By the symbol
\begin{equation}\label{interval_equation}
\dot{x}= \fA(\alpha)x + \mathfrak r(\alpha, x)
\end{equation}
we mean the family of all systems of the form
\begin{equation}\label{eq:selector}
\dot{x}= A(\alpha)x + r(\alpha, x)
\end{equation}
satisfying the following conditions:

\medskip
(i) $A: [\alpha_-,\alpha_+] \to L(d, \bbR)$ is continuous;

\medskip

(ii) $A(\alpha) \in \fA(\alpha)$ for every $\alpha \in \mathbb R$;

\medskip

(iii) $r \in \mathfrak r.$

 \medskip
 \noindent
 Denote by $\fQ$ the map from  $\mathbb R$ to the set of monic interval polynomials such that for any $\alpha \in \mathbb R$,
\begin{equation}\label{interval_characteristic equation}
\fQ(\alpha) = J_0(\alpha) + J_1(\alpha)x+\cdots+J_{n-1}(\alpha)x^{n-1}+x^n, \quad\quad\quad J_k(\alpha) = [a_k(\alpha),b_k(\alpha)],
\end{equation}
is the collection of all possible characteristic polynomials  corresponding to each member of the family $\fA(\alpha)$
(in fact, this collection constitutes an interval polynomial). To generalize Theorem \ref{thm:corollaries-all}(a,b,c) to the interval setting, we need ``interval analogs" of notations \eqref{eq:def_R_S}.
Given a family of systems \eqref{interval_equation} with interval characteristic equation \eqref{interval_characteristic equation}, put (cf.~\eqref{eq:corner-polynomials} and \eqref{rectangular-condition})

\begin{equation}\label{eq:interval-shadows}
\begin{aligned}
w_1(\alpha,\beta)&=\text{Re}(g_1(\fQ(\alpha), i\beta))\cdot\text{Re}(g_2(\fQ(\alpha), i\beta)),\\
w_2(\alpha,\beta)&=\text{Im}(h_1(\fQ(\alpha), i\beta))\cdot\text{Im}(h_2(\fQ(\alpha), i\beta)),\\
\fR(\fA)&=\{(\alpha,\beta): \alpha\in [\alpha_-,\alpha_+],w_1(\alpha,\beta)\leq0,w_2(\alpha,\beta)\leq0 \},\\
\fS_j(\fA) &= \{(\alpha,\beta) \in [\alpha_-,\alpha_+]\times\bbR: (\alpha,j\beta)\in \fR(\fA)\},\\
\fS(\fA)&= \bigcup^\infty_{j=2} \fS_j(\fA).
\end{aligned}
\end{equation}
Here $\fR$ is the set of all the purely imaginary zeros of all polynomials $P$ that belong to the family \eqref{interval_characteristic equation}.

We make the following assumptions.
\medskip

\noi\textbf{(R0)}  $r$ is continuous in both variables for any $r \in \mathfrak r$;

\medskip

\noi\textbf{(R1)} For any  $r \in \mathfrak r$, $$\lim_{\|x\| \to 0} \sup_{\alpha}{\| r(\alpha,x) \| \over \|x\|} = 0;$$

\medskip

\noi\textbf{(R2)} For any $\alpha \in [\alpha_-,\alpha_+]$,  $0\notin [a_0(\alpha),b_0(\alpha)]$;

\medskip

\noi\textbf{(R3)} $\fQ(\alpha_-)$ is $q_1$-unstable (cf.~Definition \ref{def-monic-k-stable} and \eqref{interval_equation}--\eqref{interval_characteristic equation});

\medskip

\noi\textbf{(R4)} $\fQ(\alpha_+)$ is $q_2$-unstable with $q_1\neq q_2$;

\medskip

\noi\textbf{(R5$^{\prime}$)} There exists a finite collection of disjoint sets $\cD_k \subset [\alpha_-,\alpha_+] \times \bbR_+$ such that:

(i) each $\cD_k$ is homeomorphic to a closed disk;

\smallskip

(ii) $\fR(\fA) \subset \bigcup \cD_k$;

\smallskip

(iii) $\fS(\fA) \bigcap \partial\cD_k = \emptyset$ for any $k$;

\medskip

\noi\textbf{(R5$^{\prime\prime}$)} $\fR(\fA)\cap \fS(\fA)=\emptyset$;

\medskip

\noi\textbf{(R5$^{\prime\prime\prime})$} For any  $\alpha \in [\alpha_-,\alpha_+]$ and for any  $P\in \fQ$, $P$ has at most one pair of purely imaginary roots.

We are now in a position to formulate our main result on the Hopf bifurcation in interval systems.

\begin{theorem}\label{theor:Haritonov2}
Suppose that \noi\textbf{(R0)-(R4)} hold and either \textbf{(R5$^\prime$)}, \textbf{(R5$^{\prime\prime}$)} or \textbf{(R5$^{\prime\prime\prime}$)} is satisfied.
Then, any selector \eqref{eq:selector} belonging to \eqref{interval_equation}  
has a branch of non-constant periodic solutions bifurcating  from the trivial solution.
\end{theorem}

\begin{remark}
{\rm One can easily see the parallelism between hypotheses of Theorem  \ref{theor:Haritonov2} and their non-interval counter-parts from Theorem \ref{thm:corollaries-all}(a,b,c).}
\end{remark}

\begin{remark}\label{rem:Decartes-interval-verification}
{\rm Conditions \textbf{(R3)} and \textbf{(R4)} can be verified using Kharitonov's theorem (see Theorem \ref{thm:Kharitonov}) and Lemma \ref{Interval_one_root}. To verify \textbf{(R5$^{\prime\prime\prime}$)}, one can use Lemma \ref{lem:interval-descartes}.}
\end{remark}

\section{Examples}\label{sec:examples}

Below we present three examples illustrating Theorem \ref{theor:Haritonov2} with one of the conditions
\textbf{(R5$^{\prime})$} -- \textbf{(R5$^{\prime\prime\prime})$} in each of them.
To simplify the exposition, we are dealing with higher order scalar equations
rather than with equivalent first order systems.
The class of nonlinearities $\mathfrak r$ in each example is assumed to satisfy
conditions \textbf{(R0)} and \textbf{(R1)}.

\begin{figure}[ht]
\begin{center}
\includegraphics*[width=0.35\columnwidth]{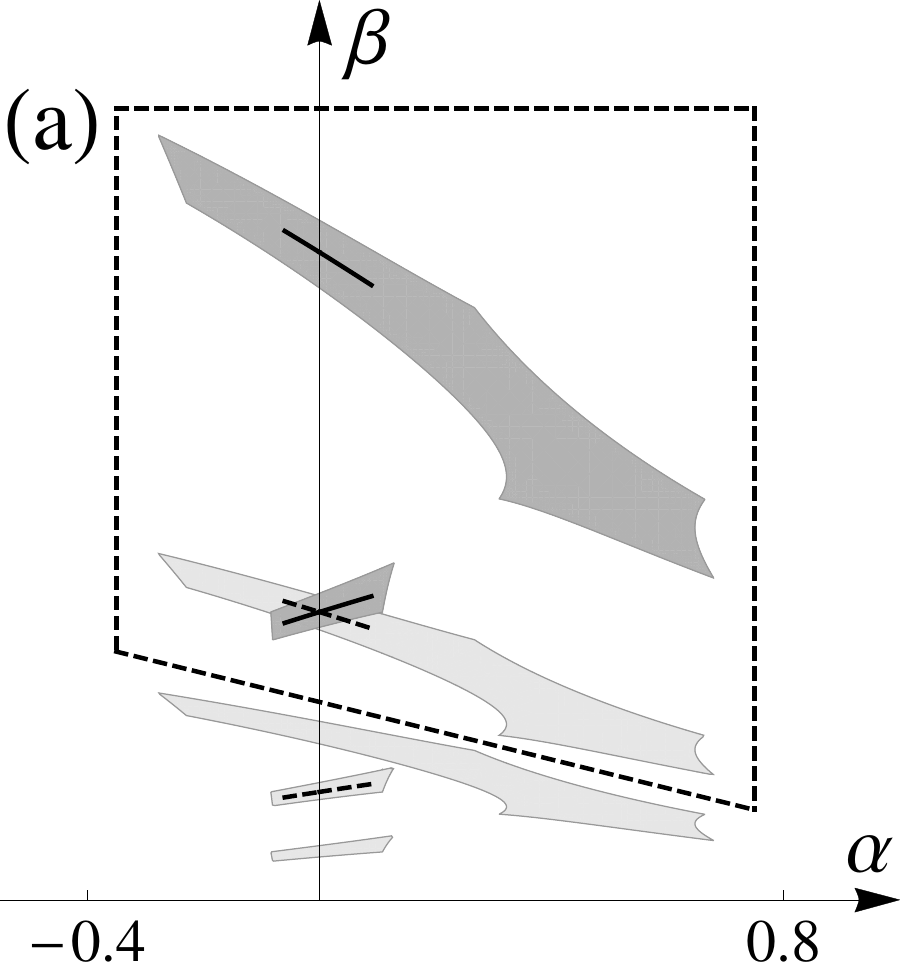} \qquad\quad \includegraphics*[width=0.33\columnwidth]{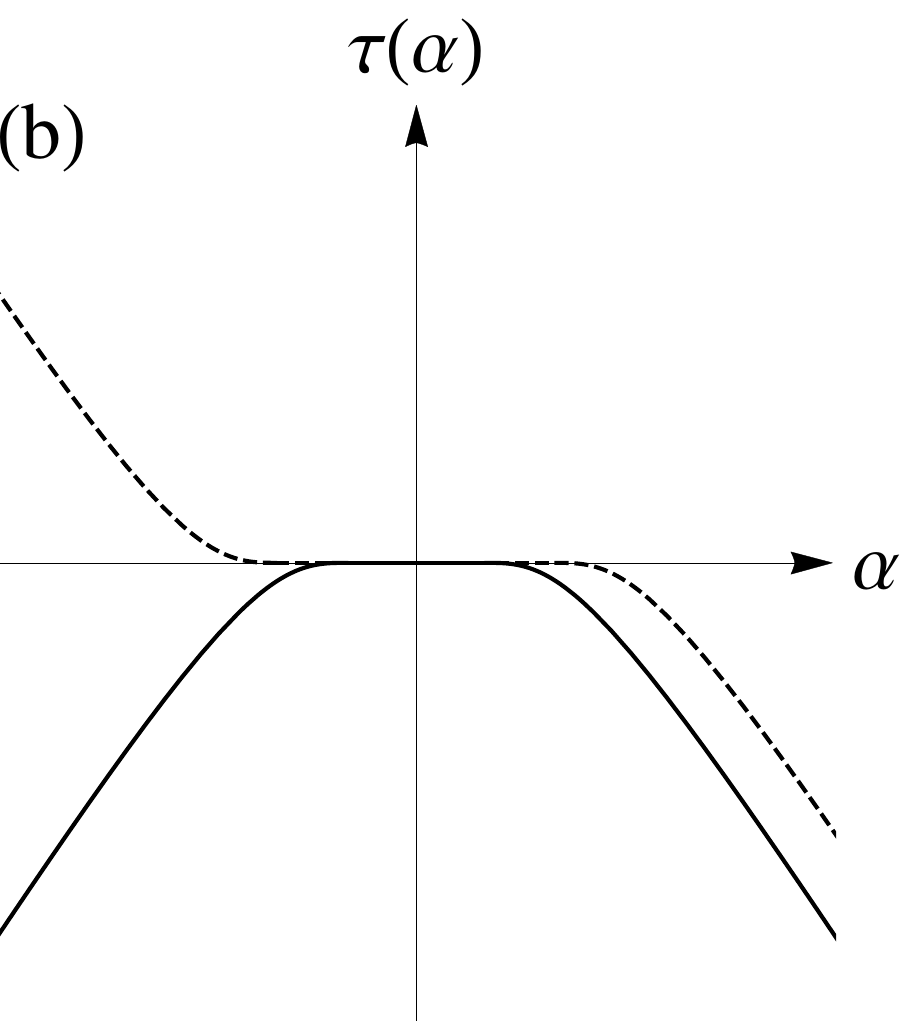}
\end{center}
\caption{
(a) The dark grey domain that consists of two connected components is the set $\fR$ 
of purely imaginary characteristic roots $i\beta$ of the interval polynomial  
(\ref{eq:fq}) for Example \ref{ex:interval1}.
The two solid black curves inside the two components of $\fR$ show the set of purely imaginary roots for a representative polynomial $P(\alpha)(\cdot)$ that belongs to the family (\ref{eq:fq}).
This representative has two purely imaginary roots $i\beta_1(\alpha)$, $i\beta_2(\alpha)$ for some interval of $\alpha$ values $[\alpha_1,\alpha_2]\subset (\alpha_-,\alpha_+)=(-0.4,0.8)$.
The light grey domains are the sets $\fS_2$ and $\fS_3$ obtained from the dark grey domain $\fR$ by the transformations $(\alpha,\beta) \mapsto (\alpha,\beta/2)$ and
$(\alpha,\beta) \mapsto (\alpha,\beta/3)$, respectively; the dashed curves inside $\fS_2$ are the images of the solid black curves in $\fR$ under this transformation.
The intersection of the solid curve and the dashed curve inside the smaller component of $\fR$ corresponds to the $2:1$ resonance $i\beta_1(\alpha)=2 i\beta_2(\alpha)$.
The dashed quadrangle $\cD_1$ contains the set $\fR$; its boundary does not intersect $\fS_i$ in accordance with \textbf{(R5$^\prime$)}.
(b) The real parts $\tau_1(\alpha)$, $\tau_2(\alpha)$ of the roots of the representative polynomial $P(\alpha,\cdot)$ that belongs to the family (\ref{eq:fq}) (schematic).
The sliding intervals $\tau_1(\alpha)=0$, $\tau_2(\alpha)=0$ correspond to the black curves $\beta_1(\alpha)$, $\beta_2(\alpha)$ shown inside the dark grey domain $\fR$ on panel (a).
}
\label{figure2}
\end{figure}

\medskip
\begin{example}[Theorem \ref{theor:Haritonov2} with \textbf{(R5$^{\prime}$)}]\label{ex:interval1}
{\rm Fix $\varepsilon= 0.28$ and, for any real $\alpha$, define four intervals as follows:
\begin{equation}\label{eq:intervals-ex1}
\begin{aligned}
J_0(\alpha)&= \{4 - 4 \alpha^2\}; \\ J_1(\alpha)&= [4\alpha - 3 \alpha^2 + \alpha^3 - \varepsilon, 4\alpha - 3 \alpha^2 + \alpha^3 + \varepsilon];\\
J_2(\alpha)&= [5 - 3\alpha + \alpha^3  - \varepsilon,  5 - 3\alpha  +  \alpha^3  + \varepsilon]; \\ J_3(\alpha)& =  \{\alpha + \alpha^2\}.
\end{aligned}
\end{equation}
Consider the following forth order interval differential equation
\begin{equation}\label{eq:Example1}
J_0(\alpha)  y + J_1(\alpha)y^{\prime} + J_2(\alpha) y^{\prime\prime} + J_3 (\alpha)y^{\prime\prime\prime} + y^{\prime\prime\prime\prime} = \mathfrak r(\alpha,y,y',y'',y'''). 
\end{equation}
The characteristic equation of the linearization of \eqref{eq:Example1} at zero has the form
\begin{equation}\label{eq:fq}
\fQ(\alpha)(\lambda) = J_0(\alpha)  + J_1(\alpha)\lambda + J_2(\alpha)\lambda^2  + J_3(\alpha)\lambda^3 + \lambda^4.
\end{equation}
Following \eqref{eq:interval-shadows}, we compute
\begin{align*}
\text{Im}(g_1(\fQ(\alpha), i\beta))&= 4 - 4 \alpha^2 - (5 - 3\alpha + \alpha^3 + \varepsilon) \beta^2 + \beta^4;\\
\text{Im}(g_2(\fQ(\alpha), i\beta))&= 4 - 4 \alpha^2 - (5 - 3\alpha + \alpha^3 - \varepsilon) \beta^2 + \beta^4;\\
\text{Re}(h_1(\fQ(\alpha), i\beta))&= (4\alpha - 3\alpha^2 + \alpha^3 - \varepsilon)\beta - (\alpha + \alpha^2)\beta^3;\\
\text{Re}(h_2(\fQ(\alpha), i\beta))&= (4\alpha - 3\alpha^2 + \alpha^3 + \varepsilon)\beta - (\alpha + \alpha^2)\beta^3.
\end{align*}
Take $\alpha_-= -0.4$ and $\alpha_+= 0.8$ and define $\fR$ and $\fS$ as in \eqref{eq:interval-shadows}. Let us show that equation \eqref{eq:Example1}
satisfies conditions of  Theorem \ref{theor:Haritonov2} with \textbf{(R5$^{\prime}$)}. Since, by construction,
$\alpha \in [-0.4,0.8]$, \textbf{(R2)}  is satisfied (cf. the first formula in \eqref{eq:intervals-ex1}). To show  \textbf{(R3)} and  \textbf{(R4)},
we use Lemma \ref{Interval_one_root}. Observe that $\fQ(\alpha)(\lambda)$ from \eqref{eq:fq} is obtained from the polynomial
$P_o(\alpha)(\lambda) = (\lambda^2 + \alpha\lambda + 1 + \alpha) \cdot (\lambda^2 + \alpha^2\lambda + 4 - 4\alpha)$ by taking
$\varepsilon$-neighborhoods of some of its coefficients. By direct verification, $P_o(\alpha_+)$ is Hurwitz stable while  $P_o(\alpha_-)$ is $2$-unstable.
To complete the verification of condition  \textbf{(R3)} (resp., \textbf{(R4)}), it remains to observe that $\{(\alpha_-,\beta) \in [\alpha_-,\alpha_+]\times \bbR_+\} \cap \fR = \emptyset$
(resp.,  $\{(\alpha_+,\beta) \in [\alpha_-,\alpha_+]\times \bbR_+\} \cap \fR = \emptyset$). These last two relations as well as condition \textbf{(R5$^{\prime}$)} are illustrated by
Figure \ref{figure2}.}
\end{example}

\begin{figure}[ht]
\begin{center}
\includegraphics*[width=0.3\columnwidth]{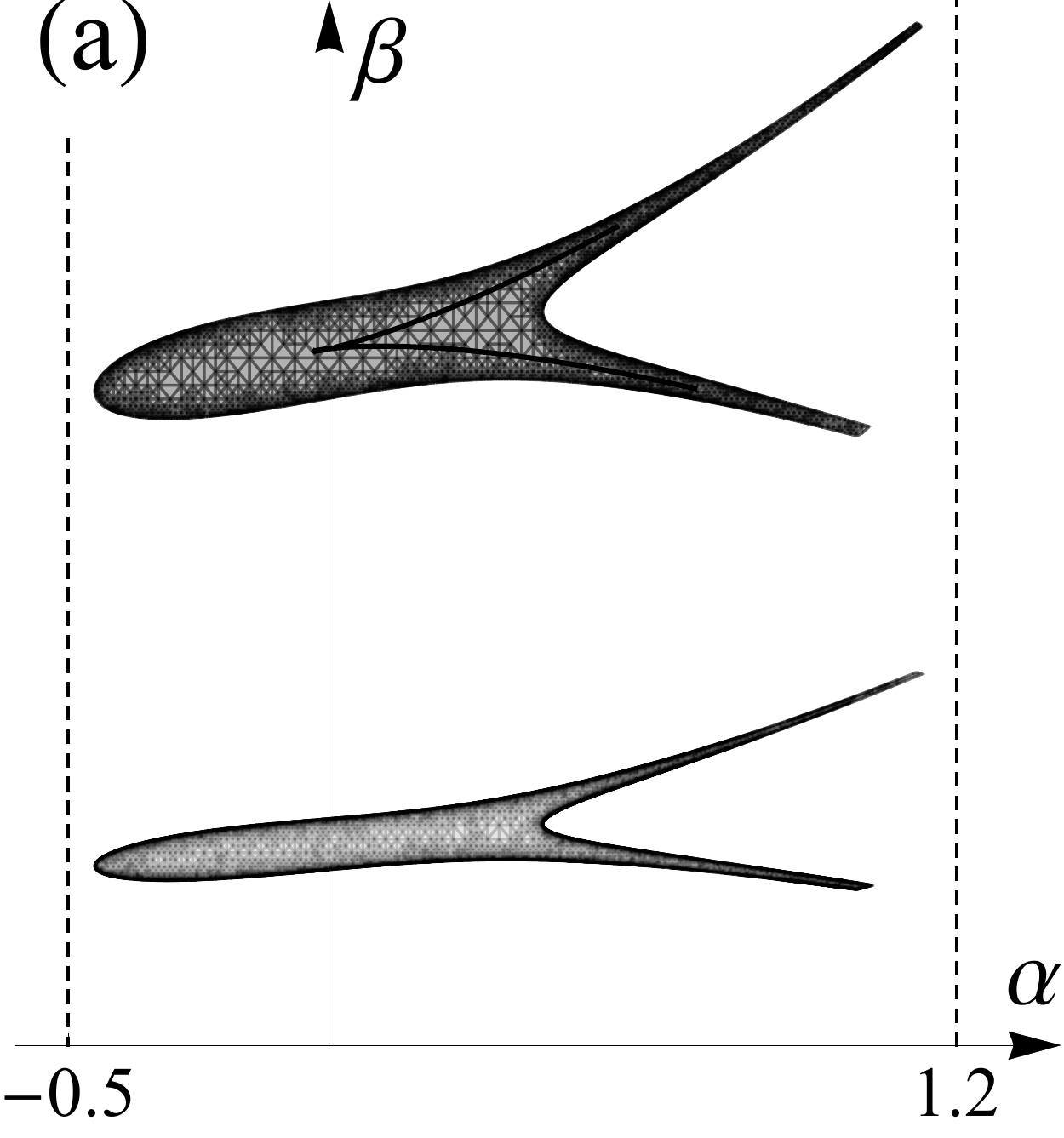} \qquad\quad \includegraphics*[width=0.28\columnwidth]{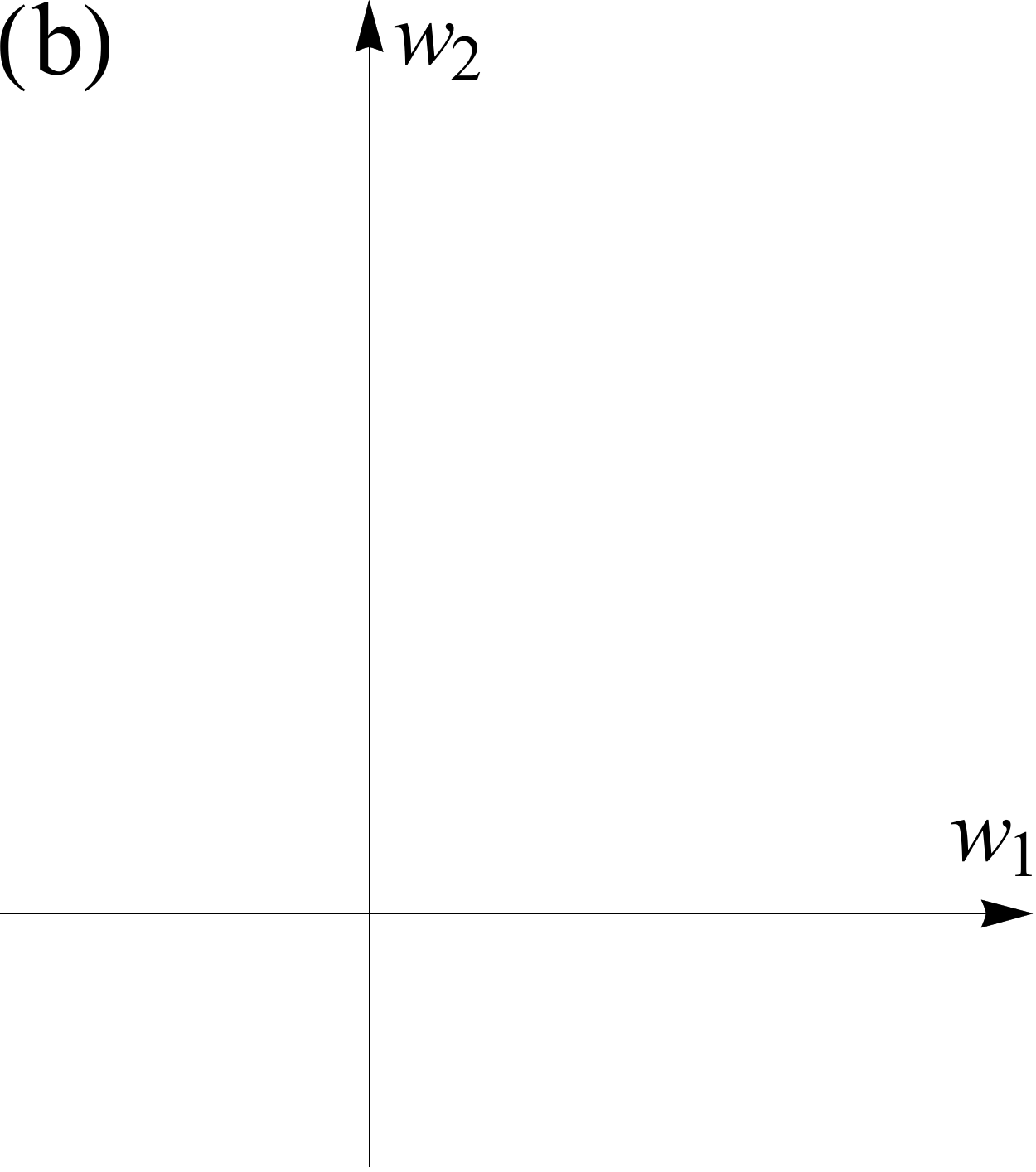}
\end{center}
\caption{(a) The nonintersecting sets $\fR$ (dark grey) and $\fS$ (light grey) for Example \ref{ex:interval2} with $[\alpha_1,\alpha_2]=[-0.5, 1.2]$.
The black curve is the set of purely imaginary roots of a representative polynomial $P(\alpha)(\cdot)$ that belongs to the family (\ref{eq:fq}).
At the corner point of this curve,  $P(\alpha)(\cdot)$ has a purely imaginary root of multiplicity 2.
The real parts of the roots of $P(\alpha)(\cdot)$ behave as shown on panel (b) of Figure \ref{figure2}. (b) Curves $(w_1(\alpha_\pm,\cdot),w_2(\alpha_\pm,\cdot))$
(thick lines) for Example \ref{ex:interval3}. The thin curve $(w_1(\alpha,\cdot),w_2(\alpha,\cdot))$ with $\alpha=0.075$ from the
interior of the interval $[\alpha_-,\alpha_+]$ intersects the negative cone $\{(w_1,w_2): w_1\le 0,\, w_2\le0\}$.}
\label{figure3}
\end{figure}

\medskip
\begin{example}[Theorem \ref{theor:Haritonov2} with \textbf{(R5$^{\prime\prime}$)}]\label{ex:interval2}
{\rm Fix $\varepsilon= 0.7$ and, for any real $\alpha$, define four intervals as follows:
\begin{equation}\label{eq:intervals-ex2}
\begin{aligned}
J_0(\alpha)&= [81 + 27\alpha + 2 \alpha^2 - \varepsilon,  81 + 27\alpha + 2 \alpha^2 + \varepsilon]; \\
J_1(\alpha)&= [9\alpha + 11 \alpha^2 + \alpha^3 - \varepsilon, 9\alpha + 11 \alpha^2 + \alpha^3 + \varepsilon];\\
J_2(\alpha)&= \{18 + 3\alpha + \alpha^3 \}; \\
J_3(\alpha)& =  [\alpha + \alpha^2 - \varepsilon, \alpha + \alpha^2 + \varepsilon].
\end{aligned}
\end{equation}
As in Example \ref{ex:interval1},  consider the interval differential equation \eqref{eq:Example1} and the characteristic polynomial
of its linearization \eqref{eq:fq}.
In this case,
\begin{align*}
\text{Re}(g_1(\fQ(\alpha), i\beta))&= (81 + 27\alpha + 2 \alpha^2 - \varepsilon) - (18 + 3\alpha + \alpha^3)\beta^2+\beta^4;\\
\text{Re}(g_2(\fQ(\alpha), i\beta))&= (81 + 27\alpha + 2 \alpha^2 + \varepsilon) - (18 + 3\alpha + \alpha^3)\beta^2+\beta^4;\\
\text{Im}(h_1(\fQ(\alpha), i\beta))&= (9\alpha + 11 \alpha^2 + \alpha^3 - \varepsilon)\beta - (\alpha + \alpha^2 + \varepsilon) \beta^3;\\
\text{Im}(h_2(\fQ(\alpha), i\beta))&= (9\alpha + 11 \alpha^2 + \alpha^3 + \varepsilon)\beta - (\alpha + \alpha^2 - \varepsilon) \beta^3.
\end{align*}
Consider the interval $[\alpha_-,\alpha_+]= [-0.5, 1.2]$. For this interval, conditions \textbf{(R2)} -- \textbf{(R4)} of Theorem \ref{theor:Haritonov2}
can be verified in the same way as in the previous example. In particular, one can use
the representative polynomial $P_o(\alpha)(\lambda)=(9 + \alpha + \alpha\lambda + \lambda^2)(9 + 2\alpha+ \alpha^2\lambda + \lambda^2)$
when proving \textbf{(R3)} and \textbf{(R4)}.
Figure \ref{figure3}a shows that condition \textbf{(R5$^{\prime\prime}$)} is also satisfied.
}
\end{example}

\medskip
\begin{example}[Theorem \ref{theor:Haritonov2} with \textbf{(R5$^{\prime\prime\prime}$)}]\label{ex:interval3}
{\rm Fix $\varepsilon= 1$. For any real $\alpha$, define five intervals
\begin{equation}\label{eq:intervals-ex3}
\begin{aligned}
J_0(\alpha)&= [36 - \varepsilon,  36 + \varepsilon]; \\ J_1(\alpha)&= [36 + 36\alpha - \varepsilon,  36 + 36\alpha + \varepsilon];\\
J_2(\alpha)&= [47  + 36\alpha - \varepsilon,  47 + 36\alpha + \varepsilon]; \\ J_3(\alpha)& =  [37 + 11\alpha -\varepsilon, 37 + 11\alpha +\varepsilon];\\
J_4(\alpha)&= [11 + \alpha -\varepsilon,   11 + \alpha +\varepsilon]
\end{aligned}
\end{equation}
and consider the fifth order interval differential equation
\begin{equation}\label{eq:Example3}
J_0(\alpha)  y + J_1(\alpha)y^{\prime} + J_2(\alpha) y^{\prime\prime} + J_3 (\alpha)y^{\prime\prime\prime} + J_4(\alpha) y^{\prime\prime\prime\prime} +y^{\prime\prime\prime\prime\prime}  =
 \mathfrak r(\alpha,y,y',y'',y''',y'''').
\end{equation}
The corresponding characteristic polynomial equals
\begin{equation}\label{eq:fq3}
\fQ(\alpha)(\lambda) = J_0(\alpha)  + J_1(\alpha)\lambda + J_2(\alpha)\lambda^2  + J_3(\alpha)\lambda^3 + J_4(\alpha)\lambda^4 + \lambda^5.
\end{equation}
Let us take $\alpha_-= -0.1$, $\alpha_+=0.09$  and show that
 conditions of  Theorem \ref{theor:Haritonov2} with \textbf{(R5$^{\prime\prime\prime}$)} are satisfied. By construction,
$\varepsilon = 1$, hence \textbf{(R2)} holds (cf. the first formula in \eqref{eq:intervals-ex3}). To show  \textbf{(R5$^{\prime\prime\prime}$)}, we apply
Lemma \ref{lem:interval-descartes}. To this end, put $Q(\lambda)= \lambda$ and $R(\lambda)= - 5$. By direct calculation,
$$
\cT(\fQ, Q,R)(\omega) = (-144 - 180\alpha \pm 6\varepsilon)\omega + (138 + 19\alpha \pm 6\varepsilon)\omega^3 +
(6 + \alpha \pm \varepsilon) \omega^5,
$$
where for brevity we denote the interval $[\mu-\varepsilon,\mu +\varepsilon]$ by $\mu \pm \varepsilon$. Since for $\alpha \in [\alpha_-,\alpha_+]$, $\cT$ has at most one sign change, property  \textbf{(R5$^{\prime\prime\prime}$)} is satisfied. Finally, to show  \textbf{(R3)} and  \textbf{(R4)}, we use the same argument as in the previous examples observing that $\mathfrak Q$  in \eqref{eq:fq3} is obtained from the polynomial $P_o(\alpha)(\lambda) = (2+ \lambda) \cdot (3+ \lambda)\cdot  (6+ \lambda) \cdot (1+ \alpha\lambda+\lambda^2)$ by taking
$\varepsilon$-neighborhoods of some of its coefficients. By direct verification, $P_o(\alpha_+)(\cdot)$ is Hurwitz stable while  $P_o(\alpha_-)(\cdot)$ is $2$-unstable.
To complete the verification of condition  \textbf{(R3)} (resp. \textbf{(R4)}), it remains to observe that the curves shown in Figure \ref{figure3}b don't intersect the negative cone $\{(x,y)\in \bbR^2\;:\;x\leq 0,y\leq 0\}$
(cf.~Lemma \ref{Interval_one_root}).
}
\end{example}

\section{Proof of Theorem \ref{thm:main-theorem}}\label{sec:main-theorem}
\subsection{Necessary condition for the Hopf bifurcation}\label{ddd}

Before proving Theorem \ref{thm:main-theorem}, let us show that the assumptions of Theorem \ref{thm:main-theorem} imply the classical necessary condition for the Hopf bifurcation.

\begin{proposition}\label{prop:Hopf-necessary}
Under the assumptions of Theorem \ref{thm:main-theorem}, there exists $(\alpha_o, \beta_o) \in \bigcup \mathcal {D}_k$ such that
$\Lambda(\alpha_o,\beta_o,0) = 0$.
\end{proposition}

\begin{proof} Observe (cf.~property \textbf{(P4)}(i)) that $\partial \cP$ is homeomorphic to a $2$-dimensional sphere.
Take the standard orientation on $\bbR^2$ and induce an orientation on $\mathcal{P}_0\subset \partial \cP$. This orientation canonically induces orientations on $\cP_{\pm}$ and the orientation on $\partial \cP$. In particular,  the local Brouwer degree for $\Lambda|_{\partial \cP} :
\partial \cP \to \mathbb C$ is correctly defined (provided that, say, the standard orientaion on $\mathbb C \simeq \mathbb R^2$
is chosen). Since $\partial \cP$ is compact and $\mathbb C$ is not compact, it follows
that $\Lambda|_{\partial \cP}$ is not surjective and, therefore,
\begin{equation}\label{eq:deg-zero}
\deg(\Lambda, \partial \cP) = 0
\end{equation}
(cf.~\cite{Dold}, Chapter VIII, Subsection 4.5).
Combining \eqref{eq:deg-zero} with condition \textbf{(P4)}(ii) and the
excision property of the local Brouwer degree, one has (cf.~\cite{AED}, p. 277):
\begin{equation}\label{eq:main-connect-cross}
\deg(\Lambda,\partial \cP) = \deg(\Lambda,\cP_+) + \deg(\Lambda,\cP_-) + \deg(\Lambda,\mathcal {P}_0) =0.
\end{equation}
By construction, the orientation on $\cP_+$ (resp., $\cP_-$) coincides with the orientation on $\mathcal {P}_0$ (resp.,
is opposite to it). Denote by  $\mathfrak t_\pm$ the number of roots of  $\Lambda(\alpha_\pm,\beta,\tau)$
in $\mathcal{P}_\pm$ (counted according to their multiplicities). It is easy to see that $\mathfrak t_\pm = \pm \deg(\Lambda,  \cP_{\pm})$.
This observation together with formula \eqref{eq:main-connect-cross}
implies
\begin{equation*}
\deg(\Lambda, \mathcal {P}_0) = \mathfrak t_- - \mathfrak t_+ \not= 0
\end{equation*}
(cf.~condition \textbf{(P4}(iii)).  On then other hand, combining condition \textbf{(P5)}(ii) with the $\mathbb Z_2$-equivariance of $\Lambda$ (see
condition \textbf{(P4)}) yields
\begin{equation}\label{eq:first-crossing}
 {\mathfrak t_- - \mathfrak t_+  \over 2} = {1 \over 2 } \deg(\Lambda, \mathcal {P}_0) =  \deg(\Lambda, \bigcup \mathcal {D}_k).
\end{equation}
By the existence property of the Brouwer degree, the conclusion follows.
 \end{proof}

 \subsection{Normalization of the period}
We are looking for periodic solutions, with unknown period $p$, of the  differential equation
$$
\dot x = f(\alpha,x).
$$
Following the standard scheme, let us introduce the unknown period $p$ as an additional parameter. Define
$
\beta= \frac{2\pi}{p}
$
and apply the change of variables
$$
u(t) = x\Big(\frac{p t}{2\pi}\Big)
$$
to obtain the system
\begin{equation}\label{eq:normal-system}
\begin{cases}
\dot u = \frac{1}{\beta}f(\alpha,u),\\
u(0) = u(2\pi).
\end{cases}
\end{equation}
We are now in a position to reformulate the original problem as an operator equation in the appropriate space of $2\pi$-periodic functions and apply the equivariant degree method.

 \subsection{$S^1$-representations}\label{subsec:S1-representation} We will
 use the first Sobolev space of  functions on the unit circle equipped with the natural structure of $S^1$-representation induced by the shift in time. Let us recall some standard facts related to $S^1$-representations. As is well-known (see, for example, \cite{BtD}), any real irreducible $S^1$-representation
is of dimension $1$ or $2$ and can be described as follows. Take an integer $l > 0$ and define the $S^1$-action on $\mathbb C \simeq \mathbb R^2$ by  $(e^{i\varphi},z) \mapsto e^{il\varphi}\cdot z$, where ``$\cdot$'' stands for complex multiplication,(denote this representation $\mathcal V_l$); also, denote by $\mathcal V_0$
the trivial one-dimensional $S^1$-representation.

Define $V=\bbR^n$. Denote by $W = H^1(S^1;V)$ the first Sobolev space of functions from $S^1$ to $V$. Observe that $W$
admits the ``Fourier decomposition"
\begin{equation}\label{eq:Fourier-decomp}
W = \overline{V \oplus \bigoplus^\infty_{l=1} W_l},
\end{equation}
where the subspace of zero Fourier modes (i.e., constant functions) is identified with $V$, while the subspace of the $l$-th Fourier modes $W_l$ is identified with the complexification of $V$ (denoted $V^c$). In particular, any function $u\in W_l$ can be written in the form $e^{ilt} \cdot (x_l + i y_l)$ for some $x_l, y_l \in V$. 
There is a natural orthogonal $S^1$-representation on $W$ given by
\begin{equation}\label{eq:S1-action}
(e^{i\varphi}, u)(t) \mapsto u(t + \varphi), \quad e^{i\varphi} \in S^1, \; u \in W.
\end{equation}
Formula \eqref{eq:S1-action} gives rise to the trivial action on $V$ and the action $(e^{i\varphi}, u)(t) \mapsto e^{il\varphi} \cdot u(t)$ on $W_l$.

\subsection{Reformulation in the functional space}

Take the first Sobolev space $W$ and define the orthogonal projector $K : W \to L^2(S^1;V)$ by
\begin{equation}\label{eq:ortho}
K(u)= {1 \over 2\pi}\int_0^{2\pi} u(t) dt, \quad u \in W.
\end{equation}
We can now rewrite \eqref{eq:normal-system} as the following operator equation in $[\alpha_-,\alpha_+]\times \bbR_+\times W$:
\begin{equation}\label{eq:operator-eq}
\mathfrak F(\alpha,\beta,u)= u - \mathcal F(\alpha,\beta,u) = 0, \quad (\alpha,\beta) \in [\alpha_-,\alpha_+]\times \bbR_+, \;\; u \in W,
\end{equation}
where
\begin{equation}\label{eq:def-calF}
\mathcal F(\alpha,\beta,u)= (L+K)^{-1}\Big({1 \over \beta} F(\alpha,u) + Ku\Big),
\end{equation}
$L : W \to  L^2(S^1;V)$ is given by $L(x)= \dot{x}$ and $F : \mathbb R \times W \to L^2(S^1;V)$ is defined by $F(\alpha,u)(t)= f(\alpha,u(t))$. Formula \eqref{eq:S1-action} gives rise to the $S^1$-action on $[\alpha_-,\alpha_+]\times \bbR_+ \times W$ (we assume that $S^1$ acts trivially on $[\alpha_-,\alpha_+]\times \bbR_+$). Moreover, it is easy to see that $\mathfrak F$ given by \eqref{eq:operator-eq}  and
\eqref{eq:def-calF} is $S^1$-equivariant.

\subsection{Reducing the problem to computing $S^1$-degree}
In order to apply the equivariant degree method, we need to localize potential bifurcating branches in a cylindric box
$\Omega \subset [\alpha_-,\alpha_+]\times \mathbb R_+ \times W$ in such a way that the operator \eqref{eq:operator-eq} is
$\Omega$-admissible. To this end, consider the sets  $\mathcal D_k$  provided by condition {\bf (P5)} and put $\Sigma_k = \partial \mathcal D_k$. Since $\Sigma_k$ is compact, there are disjoint neighborhoods $N_k$ of $\Sigma_k \subset [\alpha_-,\alpha_+]\times \bbR_+$ such that $\Lambda (\alpha, \beta, 0) \neq 0$ for all $(\alpha, \beta) \in \overline{N_k}$. Set
\begin{equation}\label{eq:tilde}
\widetilde{\cD}_k= \cD_k \cup N_k; \quad \widetilde{\Sigma}_k =\partial \widetilde{\cD}_k;\quad \widetilde{\cD}= \bigcup \widetilde{\cD}_k ; \quad \quad \widetilde{\Sigma} =\partial \widetilde{\cD}; \quad a(\alpha,\beta) = {\rm Id} - D_u\mathcal F(\alpha,\beta,0),
\end{equation}
where $D_u \mathcal F$ denotes the derivative of $\mathcal F$ with respect to $u$ (cf.~\eqref{eq:def-calF}).

\begin{lemma}\label{lem:compact-oper}
There exists  a disc $B_r(0) \subset W$ of radius $r$ centered at the origin such that
for all points $(\alpha,\beta,u) \in \widetilde{\Sigma} \times (B_r(0)\setminus \{0\})$, the following holds:
\medskip

(i) $u - \cF(\alpha,\beta,u)\neq 0$;

\medskip
(ii)  the fields $\mathfrak F(\alpha,\beta,\cdot)$ and $a(\alpha,\beta)$ are $S^1$-equivariantly
homotopic on $B_r(0)$.
\end{lemma}
\begin{proof}

(i) For a contradiction, suppose that for all $\rho > 0$, there exists $(\alpha,\beta,u) \in \widetilde{\Sigma} \times (B_{\rho}\setminus\{0\})$ with $u -\cF(\alpha,\beta,u)= 0$. Since $\widetilde{\Sigma}$ is compact, without loss of generality, assume that there exists a sequence $(\alpha_j,\beta_j,u_j)$ converging to $(\alpha_*,\beta_*,0)$ such that $(\alpha_j,\beta_j) \in \widetilde{\Sigma}$, $u_j - \cF(\alpha_j,\beta_j,u_j) = 0$ and $u_j \not=0$. Then,
$$
\frac{u_j}{\|u_j\|}-\frac{\cF(\alpha_j,\beta_j,u_j)}{\|u_j\|}=0.
$$
Observe that $(\alpha_j,\beta_j)\in \widetilde\Sigma$ implies that $\beta_j$ does not converge to $0$. Combinig this with assumption {\bf (P1)} and \eqref{eq:def-calF} yields
\begin{equation}\label{eq:dif1}
\frac{u_j}{\|u_j\|}- D_u\mathcal{F}(\alpha_j,\beta_j,0)\frac{u_j}{\|u_j\|} + \frac{r(\alpha_j,\beta_j,u_j)}{\|u_j\|}=0,
\end{equation}
where $r(\alpha_j,\beta_j,u_j) = \cF(\alpha_j,\beta_j,u_j)-D_u\mathcal{F}(\alpha_j,\beta_j,0) u$. Also,
\begin{equation}\label{eq:decomp}
D_u\mathcal{F}(\alpha_j,\beta_j,0)\frac{u_j}{\|u_j\|} = D_u\mathcal{F}(\alpha_*,\beta_*,0)\frac{u_j}{\|u_j\|} + (D_u\mathcal{F}(\alpha_j,\beta_j,0)- D_u\mathcal{F}(\alpha_*,\beta_*,0)){u_j}/{\|u_j\|}.
\end{equation}
Since $D_u\mathcal{F}(\alpha_*,\beta_*,0)$ is compact, without loss of generality, we can assume that $D_u\mathcal{F}(\alpha_*,\beta_*,0){u_j}/{\|u_j\|}$ converges to some $v_*$. In addition, keeping in mind that $D_u\mathcal{F}(\alpha_*,\beta_*,0)$ depends continuously on $\alpha,\beta$, it follows from \eqref{eq:decomp} that
$D_u\mathcal{F}(\alpha_j,\beta_j,0){u_j}/{\|u_j\|}$ converges to $v_*$. Combining this with \eqref{eq:ref-uniform-converge} and \eqref{eq:dif1}
yields that ${u_j}/{\|u_j\|}$ converges to $v_* \neq 0$. Hence (see  \eqref{eq:dif1} once again),
\begin{equation*}
v_* - D_u\mathcal{F}(\alpha_*,\beta_*,0) v_* =0
\end{equation*}
meaning that $\text{Id}-D_u\mathcal{F}(\alpha_*,\beta_*,0)$ is not invertible, which contradicts \textbf{(P5)}(iii).

\medskip

(ii) This part trivially follows from the compactness of $\widetilde{\Sigma}$ and condition \textbf{(P5)}(iii) combined with the
standard linearization argument.
\end{proof}

\bigskip

Take $\widetilde{\mathcal D}_k$ given by \eqref{eq:tilde} and $B_r(0)$ provided by Lemma \ref{lem:compact-oper}. Define
\begin{equation}\label{eq:Omega}
\Omega_k = \{ \left(\alpha,\beta, u\right) \in[\alpha_-,\alpha_+]\times \bbR_+ \times W : (\alpha,\beta) \in \widetilde{\cD}_k, \; u \in B_r(0)\},\qquad
\Omega = \bigcup \Omega_k.
\end{equation}
Clearly, $\Omega$ is $S^1$-invariant. By the existence of the invariant Urysohn function, one can take an invariant function
$\varsigma: \Omega \to \bbR$ satisfying the properties
\begin{equation}\label{eq:Urysohn-properties}
\varsigma( \alpha,\beta, 0) <0; \qquad \varsigma( \alpha,\beta, u) >0 \quad \text{for}\quad \|u\|=r.
\end{equation}
Consider the map $\fF_\varsigma: \Omega \to \bbR \oplus W$ given by
\begin{equation}\label{eq:Hopf-map}
\fF_\varsigma(\alpha,\beta,u)= (\varsigma(\alpha,\beta,u), \fF(\alpha,\beta,u)).
\end{equation}
By definition, any solution to the equation $\fF_\varsigma(\alpha,\beta,u)=0$ is also a solution to \eqref{eq:normal-system}.
In addition, $\fF_\varsigma$ is an $S^1$-equivariant $\Omega$-admissible map for which $\sdeg(\fF_\varsigma, \Omega)$ is correctly defined.

\begin{remark}\label{rem:independ-Urysohn}
{\rm As long as an invariant Urysohn function $\varsigma$ satisfies properties \eqref{eq:Urysohn-properties}, $\sdeg(\fF_\varsigma,\Omega)$ is independent of
the choice of $\varsigma$ (homotopy property of the $S^1$-degree).}
\end{remark}

\medskip

The next statement provides a sufficient condition for the existence of a branch of periodic solutions
bifurcating from the trivial solution
(cf.~Definition \ref{YTM}). We follow the scheme suggested in \cite{AED} (see Theorem 9.18) with several modifications making the argument more transparent.
\begin{proposition}\label{thm:ode18}
Given system (\ref{eq:sys-HB}), assume conditions \textbf{(P0) -- (P5)} are satisfied.
Take $\Omega$ defined by \eqref{eq:Omega} and $\frak F_\varsigma$ defined by (\ref{eq:Hopf-map}). Assume $\s1deg(\frak F_\varsigma,\Omega)\not=0$.
Then, system (\ref{eq:sys-HB}) has a branch of periodic solutions
bifurcating from the trivial solution.
\end{proposition}

\medskip
\noindent
As in \cite{AED}, the following statement is the main topological ingredient in the proof of Proposition \ref{thm:ode18} (cf. Theorem 3 in \cite{Kurat}, p. 170).

\begin{proposition}[Kuratowski]\label{prop:Kurat}
Let $X$ be a metric space, $A,B \subset X$ two disjoint closed sets in $X$, and $K$ a compact set in $X$ such that $K \cap A \not= \emptyset\not=K \cap B$.
If the set $K$ does not contain a connected component $K_o$ such that $K_o \cap A \not= \emptyset\not=K_o \cap B$, then there exist two disjoint open sets
$V_1$, $V_2$ such that $A \subset V_1$,  $ B \subset V_2$ and $A \cup B \cup K \subset V_1 \cup V_2$.
\end{proposition}

\noi{\bf Proof of Proposition \ref{thm:ode18}}. Put $$K=\fF^{-1}(0)\cap \overline\Omega.$$ Consider the family of invariant functions $\varsigma_q:\overline\Omega\to \bbR$
given by $$\varsigma_q(\alpha,\beta,u)= \|u\|-q,\; \quad 0\leq q \leq r.$$ Suppose for contradiction, there does not exist a compact connected set $K_o \subset K$ with $K_o \cap \varsigma_0^{-1}(0) \not= \emptyset \not= K_o \cap \varsigma_r^{-1}(0)$. To apply Proposition \ref{prop:Kurat}, we need to show that $K \cap \varsigma_0^{-1}(0) \not= \emptyset \not= K \cap \varsigma_r^{-1}(0)$. Notice that for  any $q\in (0,r)$, $\varsigma_q$ satisfies properties \eqref{eq:Urysohn-properties}, so $\sdeg(\fF_{\varsigma_q},\Omega \neq 0$ (cf Remark \ref{rem:independ-Urysohn} and the assumptions of Proposition \ref{thm:ode18}. By the existence property of $S^1$-degree, for each $q\in (0,r)$ there exists $(\alpha_q,\beta_q,u_q)\in K$ with $\|u_q\|=q.$ Since K is compact, it follows that there exist $(\alpha_0,\beta_0,u_0)\in K\cap \varsigma_0^{-1}(0)$ and $(\alpha_r,\beta_r,u_r)\in K\cap \varsigma_r^{-1}(0)$.

Put $$A^\prime= \{(\alpha,\beta,u)\in \overline\Omega \;:\; \|u\|=0\}, \quad A^{\prime\prime}= \{(\alpha,\beta,u)\in \overline\Omega \;:\; \|u\|=r\}.$$
 Then, by Proposition \ref{prop:Kurat}, there exist open {\it disjoint} sets $N^\prime \supset A^\prime$, $N^{\prime\prime}\supset A^{\prime\prime}$ with $K\subset N^\prime\cup N^{\prime\prime}$. Put
\begin{equation}\label{eq:Z-prime}
Z^\prime = \overline{A^\prime\cup(K\cap N^\prime)}; \quad  Z^{\prime\prime}= \overline{A^{\prime\prime}\cup(K\cap N^\prime)}
\end{equation}
and let us, first, show that $Z^\prime$ is $S^1$-invariant. Notice that $A^\prime$ is invariant. Suppose for contradiction that $K\cap N^\prime$  is not invariant. Then, there exist $u\in K\cap N^\prime$ and $\gamma \in S^1$ such that $(\gamma,u) \notin  K\cap N^\prime$. However, since $K$ is invariant and $K\subset N^\prime\cup N^{\prime\prime}$, it follows that $(\gamma, u) \in K\cap N^{\prime\prime}$. We now have $S^1(u)\subset N^\prime\cup N^{\prime\prime}$ with $S^1(u)\cap N^\prime \not= \emptyset \not= S^1(u)\cap N^{\prime\prime}$
and $N^{\prime} \cap N^{\prime\prime} = \emptyset$, which contradicts the connectedness of $S^1(u)$.  Thus, $Z^\prime$ is invariant as the union of invariant sets.
Similarly, $Z^{\prime\prime}$ is also invariant.

Next, define an invariant Urysohn function $\mu:\overline\Omega \to \bbR$ with the following property:
\begin{equation}\label{eq:mu-function}
\mu(\alpha,\beta,u)=\begin{cases}
1, \quad {\rm if} \; (\alpha,\beta, u)\in Z^{\prime};\\
0, \quad  {\rm if} \; (\alpha,\beta, u)\in Z^{\prime\prime}.
\end{cases}
\end{equation}
Take $\varsigma: \overline\Omega \to \bbR$ defined by
\begin{equation}\label{eq:sigma}
\varsigma(\alpha,\beta,u)= \|u\| - \mu(\alpha,\beta,u)\cdot r.
\end{equation}
Clearly, $\varsigma$ is invariant and satisfies properties  \eqref{eq:Urysohn-properties} (cf Remark \ref{rem:independ-Urysohn} and the assumptions of Proposition \ref{thm:ode18}) so $\sdeg(\fF_\varsigma,\Omega)\neq 0$.  By the existence property of the $S^1$-degree, there exists $(\alpha_\ast, \beta_\ast,u_\ast)$ with $\|u\|- \mu(\alpha_\ast,\beta_\ast, u_\ast)=0$.
Since $K \subset N^\prime \cup N^{\prime\prime}$ and  $N^\prime \cap N^{\prime\prime} = \emptyset$, it follows that either $(\alpha_\ast,\beta_\ast, u_\ast))\in N^{\prime}$ or $(\alpha_\ast,\beta_\ast, u_\ast)\in N^{\prime\prime}$. Assume $(\alpha_\ast,\beta_\ast, u_\ast)\in N^{\prime}$.  Then (cf. \eqref{eq:Z-prime},  \eqref{eq:mu-function}) and \eqref{eq:sigma}),
$\|u_{\ast}\|=r$,  i.e. $(\alpha_\ast,\beta_\ast,u_\ast)\in A^{\prime\prime}\cap N^{\prime}=\emptyset$. Similarly, the assumption
$(\alpha_\ast,\beta_\ast,u_\ast)\in N^{\prime\prime}$ leads to a contradiction.


%

\subsection{Computation of  $\sdeg(\frak F_\varsigma,\Omega)$ via deformations}

Proposition \ref{thm:ode18} reduces the proof of Theorem \ref{thm:main-theorem} to the computation of $\s1deg(\frak F_\varsigma,\Omega)$ and showing that this degree is non-zero. Our goal now is to connect $\s1deg(\frak F_\varsigma,\Omega)$  to spectral properties of $D_xf(\alpha,0)$
(cf.~condition {\bf (P1)}). This will be done in several steps.

\medskip
{\it Step I: Reduction to a circle.} Put $\widetilde{\frak F}_\varsigma(\alpha,\beta,u) = (\varsigma(\alpha,\beta,u), a(\alpha,\beta)u)$ (cf.~\eqref{eq:tilde}).
Since $\varsigma(\alpha,\beta,0) < 0$, it follows from Lemma \ref{lem:compact-oper}(ii)
that $\widetilde{\frak F}_\varsigma$ is $S^1$-equivariantly homotopic to  $\frak F_\varsigma$ on $\Omega_k$.

Take $\widetilde{D}_k$ and $N_k$ from \eqref{eq:tilde} and assume, without loss of generality, that  $N_k$ is homeomorphic to $[-1,1]\times S^1$. Let $(\xi_k,\gamma_k) : N_k \to [-1,1]\times S^1$ be a trivialization
taking $\widetilde{\Sigma}_k$ to $\{1\} \times S^1$.
For any $k$, define three functions $g_k: \widetilde{\cD}_k \to \bbR$, $\widetilde\varsigma_k: \Omega_k \to \bbR$ and $\widetilde\fF_{1k}: \Omega_k  \to \bbR \oplus W$ by
\begin{equation}\label{eq:def-g}
g_k(\alpha,\beta)= \begin{cases} 0,& (\alpha,\beta) \in \widetilde{\cD}_k \setminus N_k;\\
\xi_k(\alpha,\beta) + 1, & (\alpha,\beta) \in N_k;
\end{cases}
\end{equation}
\begin{equation}\label{eq:def-til-sigma}
\widetilde\varsigma_k(\alpha,\beta,u)= g_k(\alpha,\beta)(\|u\|-r)+\|u\| + r;
\end{equation}
\begin{equation}
\fF_{k}^1(\alpha,\beta,u) = (\widetilde\varsigma_k(\alpha,\beta,u), a(\alpha,\beta)u).
\end{equation}
Obviously, the boundary $\partial \Omega_k$ of the domain $\Omega_k$ consists of three pieces:
$$
\partial \Omega_k = \{\|u\| = r\} \; \cup \; \{u = 0, \, \xi_k(\alpha,\beta) = 1\} \; \cup \; \{u\neq 0,\  \xi_k(\alpha,\beta) = 1\}.
$$
On the first piece, $\varsigma_k$ and $\widetilde\varsigma_k$ are both positive, while on the second piece they are both negative.
Also, on the third piece $a(\alpha,\beta)u$ is non-zero. Hence, the vector fields $\fF_{k}^1$ and $\widetilde{\fF}_{\varsigma}$ are not directed oppositely on $\partial \Omega_k$, therefore they are equivariantly homotopic on $\overline{\Omega}_k$.
Define $\Omega_k^1 = N_k \times B_r(0)$. By \eqref{eq:def-g} and \eqref{eq:def-til-sigma}, for all $(\alpha,\beta) \in \widetilde{\cD}_k \setminus N_k$, one has
$\widetilde\varsigma_k > 0$. Hence, by the excision and homotopy properties of the $S^1$-degree (see Appendix),
$\sdeg(\fF_{\varsigma},\Omega_k)=\sdeg(\fF_{k}^1,\Omega_k^1)$.

Define $\eta_k : N_k \to \Sigma_k$ by
$\eta_k(\alpha,\beta) = (\xi_k,\gamma_k)^{-1}(0,\gamma_k(\alpha,\beta))$ and
$\tilde{\fF}_k : \Omega_k^1\to \bbR \oplus W$ by
\begin{equation}\label{eq:def-tildeF}
\widetilde{\fF}_k(\alpha,\beta,u) = (\widetilde{\varsigma}_k(\alpha,\beta,u), a(\eta_k(\alpha,\beta))u).
\end{equation}
Then, by the homotopy property of the $S^1$-degree,
\begin{equation}\label{eq:deg-tildeF}
\sdeg(\fF_\varsigma,\Omega_k)=\sdeg(\widetilde{\fF}_k,\Omega_k^1).
\end{equation}
Observe that formulas \eqref{eq:def-tildeF}, \eqref{eq:deg-tildeF} reduce the computation of $\sdeg(\fF_\varsigma,\Omega_k)$
to studying $S^1$-equivariant homotopy properties of restrictions of $a_k : N_k \to GL^{S^1}_c(W)$ to the zero section
$\Sigma_k \simeq \{0\} \times S^1$, where  $GL^{S^1}_c(W)$ stands for the group of $S^1$-equivariant linear completely continuous vector fields in $W$.

\medskip
{\it Step II: Computation of the degree.}
For any $m \in \mathbb N$, put $W^m= V \oplus W_1 \oplus ... \oplus W_m$ (cf.~\eqref{eq:Fourier-decomp}).
Combining the compactness of the operator $a$ with the suspension property of the $S^1$-equivariant degree (see Appendix), one can find a sufficiently large $m$ such that the field $\widetilde{\fF}_k$ is equivariantly homotopic to the compact field
$\mathfrak F_k^2 : \overline{\Omega}_k^1 \to \mathbb R \oplus W$ defined  by
\begin{equation}\label{eq:finite-approx}
\mathfrak F_k^2(\alpha,\beta,u) = (\widetilde{\varsigma}_k(\alpha,\beta,u), \widetilde{a}(\eta_k(\alpha, \beta))u),
\end{equation}
where
$\widetilde{a}_k(\alpha,\beta)=  a_k(\alpha,\beta)|_{W^m} + {\rm Id} |_{(W^m)^{\perp}}$. Put
\begin{equation}
\widetilde{a}_k^0(\alpha,\beta):= \widetilde{a}_k(\alpha,\beta)|_V; \quad  \widetilde{a}_k^l(\alpha,\beta):= \widetilde{a}_k(\alpha,\beta)|_{W_l}, \; l > 0.
\end{equation}
Fix some $\overline{\alpha}$ between $\alpha_-$ and $\alpha_+$. By condition {\bf (P3)}, the map $\widetilde{a}_k^0 : N_k \to GL(V)$ is homotopic to the constant map
$\overline{a} : N_k \to GL(V)$ given by  $\overline{a}(\alpha,\beta) \equiv -D_xf(\overline{\alpha},0)$. Now, we are going to use formula \eqref{thm:main-computational-Theorem} presented in Appendix.
To this end, one needs to separate the ``contribution" of the zero Fourier mode to the $S^1$-degree from other modes. Define
\begin{equation}
\Omega_k^{2} = \Omega_k^1 \cap \big(\mathbb R^2 \oplus W^m), \quad  \Omega_k^{\ast}= B \times \Omega_k^{2},
\end{equation}
where $B$ is the unit ball in $V$. Also, define $\hat{\mathfrak F} _k : \overline{\Omega_k^2} \to \mathbb R \oplus W$
and $\mathfrak F^{\ast}_k : \overline{\Omega_k^{\ast}} \to V \oplus \big(\mathbb R \oplus W\big)$ by
\begin{equation}
\hat{\mathfrak F}_k (\alpha,\beta,u)= (\widetilde{\varsigma}_k(\alpha,\beta,u), \bigoplus_{l = 1}^m \widetilde{a}_k^l(\eta_k(\alpha,\beta))u), \quad
\mathfrak F^{\ast}_k = \overline{a}  \times  \hat{\mathfrak F}_k.
\end{equation}
Combining the suspension property of the $S^1$-degree with the product formula (see \cite{AED}, Theorem 6.8), one obtains
\begin{equation}\label{eq:suspension-product}
\sdeg(\mathfrak F_{\varsigma},\Omega_k) = \sdeg(\mathfrak F^{\ast}_k,\Omega_k^{\ast}) = \sign (\det(\overline{a})) \cdot
\sdeg(\hat{\mathfrak{F}}_k, \Omega_k^2).
\end{equation}
Further, by applying  formula \eqref{thm:main-computational-Theorem},
\begin{equation}
\sdeg(\mathfrak F_{\varsigma},\Omega_k) = \sign (\det(\overline{a})) \cdot \sum_{l=1}^m \big(\deg(\text{det}_{\mathbb C} (\widetilde{a}_k^l), \widetilde\cD_k) (\bbZ_l)\big).
\end{equation}
Finally, applying the additivity property of $\sdeg$ and the Brouwer degree, we get
\begin{equation}\label{eq:rev-in-science}
\begin{aligned}
 \sdeg(\mathfrak F_{\varsigma},\Omega) &= \sign (\det(\overline{a})) \cdot \sum_k \sum_{l=1}^m \big(\deg(\text{det}_{\mathbb C} (\widetilde{a}_k^l), \widetilde\cD_k) (\bbZ_l)\big)\\
 &= \sign (\det(\overline{a})) \cdot \sum_{l=1}^m \big(\deg(\text{det}_{\mathbb C} (a_l), \widetilde\cD) (\bbZ_l)\big),
\end{aligned}
\end{equation}
where $a_l(\alpha,\beta)= a(\alpha,\beta)|_{W_l}$.

\medskip

{\it Step III: Reduction to crossing numbers.} Observe (see \cite[p.~266]{AED}) that
$$
a_l(\alpha,\beta)= \frac{1}{il\beta}[il\beta\,\id -D_xf(\alpha,0)].
$$
Put $a'_l(\alpha,\beta)=il\beta \cdot a_l(\alpha, \beta)$. Since $\beta >0$, the map $a'_l$ is homotopic to $a_l$. Note (cf.~\eqref{def:Lambda}) that
$\text{det}_{\mathbb C}(a'_1(\alpha,\beta))=\Lambda(\alpha,\beta,0)$. Finally (cf.~condition \textbf{(P4)}(iii) and \eqref{eq:first-crossing}),
$$
\deg(\text{det}_{\mathbb C} (a_1), \widetilde\cD) = \frac{\mathfrak{t}_--\mathfrak{t}_+}{2}\not=0.
$$
Hence (cf.~\eqref{eq:rev-in-science}), $\sdeg(\mathfrak F_{\varsigma},\Omega) \neq 0$. The application of Proposition \ref{thm:ode18} completes the proof.

\section{Proof of Theorems \ref{thm:corollaries-all} and \ref{theor:Haritonov2}}\label{sec:proofs-remaining}

\subsection{Proof of Theorem \ref{thm:corollaries-all}(a,b,c)}
(a) Our goal is to construct a domain $\cP$ satisfying \textbf{(P4)} in such a way that \textbf{(P5$^{\prime}$)} would imply \textbf{(P5)}.
To this end, take $\cD_k$ provided by {\bf(P5$^\prime$)} and $\alpha_-$, $\alpha_+$ provided by \textbf{(P4$^{\prime}$)}. Next, take a sufficiently large $M > 0$ to ensure that
\begin{equation}\label{eq:number-M}
\Lambda^{-1}(0) \subset \{[\alpha_-,\alpha_+] \times B_M(0)\}=:B,
\end{equation}
where $B_M(0)$ stands for the closed ball of radius $M$ centered at the origin in the $(\beta,\tau)$-plane.
Also due to compactness, there exists a $\delta > 0$ such that
\begin{equation}\label{eq:delta}
\Lambda^{-1}(0) \cap\{0 \leq  \tau \leq \delta\} \subset {\rm int} (\bigcup \cD_k) \times \{0 \leq \tau \leq\delta\}.
\end{equation}
Define
\begin{equation}\label{eq:P}
\cP= \Big(B \cap \{\tau \geq \delta\}\Big) \cup \bigcup \cD_k \times \{0 \leq \tau \leq\delta\}.
\end{equation}
Since $B \cap  \cD_k \times \{0 \leq \tau \leq\delta\}$ is homeomorphic to a disc, $\cP$  satisfies \textbf{(P4)}(i). By the choice of $M$ and $\delta$
(see \eqref{eq:number-M} and \eqref{eq:delta}), $\cP$  satisfies \textbf{(P4)}(ii). Also, \textbf{(P4)$^{\prime}$} guarantees \textbf{(P4)}(iii). Finally, by construction,
 $\cP$ and $\cD_k$ satisfy {\bf(P5)}(ii).

\medskip
\noindent
(b) To prove Part (b), it suffices to deduce \textbf{(P5$^{\prime}$)} from \textbf{(P5$^{\prime\prime}$)}.
Notice that $R(f)$ is the set of roots of polynomials with coefficients parameterized  by $\alpha \in [\alpha_-,\alpha_+]$. Hence, the coefficients of these polynomials are uniformly bounded. Observe also that the leading coefficient of these polynomials is identically equal to $1$, therefore $R(f)$
is a compact set.

For any $\varepsilon > 0$, there exists a sufficiently large $m$
 such that $\bigcup^\infty_{k=m}  S_i(f)\subset [\alpha_-,\alpha_+] \times \{\beta <  \varepsilon\} $.
  Since $[\alpha_-,\alpha_+]$ is compact and
  $D_xf(\cdot,0)$ is non-singular,
 it follows that $R(f)$ is uniformly separated from $[\alpha_-,\alpha_+] \times \{ \beta <  \varepsilon\} $ provided that
$\varepsilon$ is small enough. On the other hand, the sets $\bigcup^{m-1}_{k=2}  S_i(f)$ and $R(f)$ are
compact and disjoint (see condition \textbf{(P5$^{\prime\prime}$)}), so they can be uniformly separated.
Hence there  exists a  neighborhood $N_\varepsilon(R(f))$ of $R(f)$ in $[\alpha_-,\alpha_+] \times \mathbb R_+$ such that
$N_\varepsilon(R(f))\cap S(f) = \emptyset$. Without loss of generality, one can assume that $N_\varepsilon(R(f))$ is a finite union
of discs, therefore, the complement to $N_\varepsilon(R(f))$ in $[\alpha_-,\alpha_+]\times \bbR_+$ has finitely many bounded connected components, say, $\{U_i\}_{i=1}^K$. Set
$$
\cD= N_\varepsilon(R(f)) \cup \bigcup_{i=1}^K U_i.
$$
By construction, $\overline{\cD}$ is a finite collection of disjoint sets homeomorphic to closed discs (denoted $\cD_k$) and $\partial \cD_k \subset \partial N_\varepsilon(R(f))$, thus
$\cD_k$ satisfies condition \textbf{(P5$^{\prime}$)}. Hence, the result follows from Theorem \ref{thm:corollaries-all}(a).

\medskip
\noindent
(c) To prove part (c), it suffices to deduce \textbf{(P5$^{\prime\prime}$)} from \textbf{(P5$^{\prime\prime\prime}$)}. To this end, assume, by contradiction, that  \textbf{(P5$^{\prime\prime}$)} is not satisfied. Then, there exist a point $(\alpha,\beta) \in [\alpha_-,\alpha_+]\times \bbR_+$ and an integer $k\ge 2$ such that $(\alpha,\beta), (\alpha, k \beta) \in R(f)$. This contradicts \textbf{(P5$^{\prime\prime\prime}$)}.

\subsection{Proof of Theorem \ref{thm:corollaries-all}(d)}
Let $\cP$ be the set provided by condition {\bf (P4)}. Our first goal is to construct $\cP^{\prime} \supset \cP$ such that (a) $\cP^{\prime}$ satisfies {\bf (P4)}; and, (b)
$\cP^{\prime}_0$  is a disjoint union of finitely many sets homeomorphic to a closed disc (cf.~\eqref{eq:P-0-set}). To this end,  without loss of generality (use a small perturbation of $\cP$ if necessary), one can assume that $\cP_0 \subset [\alpha_-,\alpha_+]\times \bbR_+$ is a disjoint union $\cup_{i=1}^m K_i$,  where $K_i$ is a $(\nu_i + 1)$-connected compact domain. Using the same surgery argument as in the proof of Alexander's tame sphere Theorem (see, for example, \cite{Calegari}, Theorem 4.34), one can construct $\cP^{\prime}$ satisfying (a) and (b).

Our next goal is to construct a finite collection of discs $ \cD_k \subset [\alpha_-,\alpha_+]\times \bbR_+$ satisfying {\bf (P5)}.  Take $R$ and $S$ given by
\eqref{eq:def_R_S}. Using the same argument as in  the proof of Theorem \ref{thm:corollaries-all}(b) above, one can construct a sufficiently small neighborhood
$N_{\varepsilon}(R \cap \cP^{\prime}_0)$ of the intersection $R \cap \cP^{\prime}_0$ such that $N_{\varepsilon}(R \cap \cP^{\prime}_0) \cap S = \emptyset$ and $N_{\varepsilon}(R \cap \cP^{\prime}_0) \subset \cP_0^{\prime}$
(cf.~condition ${\bf (P5^{\prime\prime})}$). Take $ C= [\alpha_-,\alpha_+]\times \bbR_+ \setminus \overline{N_{\varepsilon}(R \cap \cP^{\prime}_0)}$. By the standard compactness argument, without loss of generality, assume that  $C$ splits into finitely many connected components $C = C_0 \cup C_1 \cup ... \cup C_r$, where $C_0$ stands for the (unique) unbounded component.
Put $\cD= \cup_{i=1}^rC_i \cup \overline{N_{\varepsilon}(R \cap \cP^{\prime}_0)}$. Let us show that $\cD$ is a finite union of discs. By construction, $\cD = \cup_{i=1}^k\cD_k$ is a finite disjoint union of regular closed subsets (i.e., each  $\cD_k$ is a closure of its interior). To show that each $\cD_k$ is contractible, take a closed curve $\gamma \subset \cD$ and assume that it is not contractible to a point inside $\cD_k$. Then, there exists a set $K \subset [\alpha_-,\alpha_+]\times \bbR_+ \setminus \cD_k$ bounded by $\gamma$. However, this contradicts the construction of $\cD$. Therefore, $\cD_k$ satisfies condition {\bf (P5)}(i). Also, since $\partial \cD_k \subset \partial N_{\varepsilon}(R \cap \cP^{\prime}_0)$ for any $k$, it follows that
$\cD_k$ satisfies {\bf (P5)}(iii). Finally, to show that $\cD_k$ satisfies {\bf (P5)}(ii), observe that $R \cap \cD \supset R\cap \cP_0^{\prime}$. Since, $\cD \subset \cP_0^{\prime}$, one has
$R \cap \cD = R\cap \cP_0^{\prime}$ and {\bf (P5)}(ii) follows.

\subsection{Proof of Theorem \ref{theor:Haritonov2}}
Clearly, if \eqref{interval_equation} satisfies \textbf{(R0)-(R4)}, then any  selector \eqref{eq:selector} belonging to \eqref{interval_equation} satisfies
 \textbf{(P0)-(P4)}. Similarly, if \eqref{interval_equation} satisfies \textbf{(R5$^{\prime}$)} (resp.,   \textbf{(R5$^{\prime\prime}$, \textbf{(R5$^{\prime\prime\prime}$)}}),
then any  selector \eqref{eq:selector} belonging to \eqref{interval_equation} satisfies \textbf{(P5$^{\prime}$)} (resp.,   \textbf{(P5$^{\prime\prime}$, \textbf{(P5$^{\prime\prime\prime}$)}}).
The result follows.

\section{Appendix: $S^1$-degree}

Let $G$ be a compact Lie group acting on a metric space $X$ (see, for example, \cite{Bre}). For any $x \in X$, put $G(x)=\{gx \in X \; : \; g \in G\}$ and call it the orbit of $x$.
 A set $Z \subset X$ is called $G$-invariant (in short, invariant) if it contains all its orbits.
Assume $G$ acts on two metric spaces $X$ and $Y$. A continuous map $f : X \to Y$ is called $G$-equivariant
if $f(gx) = gf(x)$ for all $x \in X$ and $g \in G$. In particular, if the action of $G$ on $Y$ is trivial, then the equivariant map is called $G$-invariant. We refer to \cite{Bre,tD,AED}
(resp. \cite{BtD,GolSchSt,GS,AED}) for the equivariant topology (resp. representation theory) background frequently used in the present paper.

\medskip

Let $V$ be an orthogonal $S^1$-representation. Suppose that an open bounded invariant set $\Omega \subset \bbR \oplus V$ is invariant with respect to the $S^1$ action, where we assume that $S^1$ acts trivially on $\bbR$. We say that an equivariant map $f: \overline{\Omega} \to V$ is admissible if $f^{-1}\{0\} \cap \partial \Omega = \emptyset$. In this case,  $(f,\Omega)$ is called an admissible pair. Similarly, a continuous map $h : [0,1] \times \overline{\Omega} \to V$ is called an admissible (equivariant) homotopy if $h(t,\cdot)$ is admissible for any $t \in [0,1]$. It is possible to axiomatically define a unique function $\sdeg$ which assigns to each admissible pair a formal sum of finite cyclic groups with integer coefficients (cf.~\cite{AED}, pp. 109, 113 ). The following is a partial list of the axioms:

\medskip

\noi\textbf{(A1) (Existence)} If $\sdeg(f,\Omega) = \sum_{l=1}^k n_{l_k}(\bbZ_{l_k})$ and $n_{l_k} \neq 0$ for some $k$, then there exists an $x\in \Omega$ such that $f(x) = 0$ and
$\mathbb Z_{l_k} \subset G_x$.

\medskip

\noi\textbf{(A2) (Homotopy)} Suppose that $h:[0,1]\times \overline{\Omega} \to V$ is an admissible equivariant homotopy; then,
$\sdeg(h(t,\cdot,\cdot),\Omega) = const$.

\medskip

\noi\textbf{(A3)(Additivity)} For two invariant open disjoint subsets $\Omega_1,\Omega_2 \subset \Omega$ with $f^{-1}(0)\cap \Omega\subset  \Omega_1\cup\Omega_2 $, $\sdeg(f,\Omega)= \sdeg(f,\Omega_1)+ \sdeg(f,\Omega_2)$.

\medskip
\noi
\textbf{(A4)(Normalization)} Take $\mathcal V_1$ (cf.~Subsection \ref{subsec:S1-representation}) and define the set $\Omega_0$ and map $b : \mathbb R \oplus \mathcal V_1 \to \mathcal V_1$ by
$$
\Omega_0 = \Big\{(t,z) \in \mathbb R \oplus \mathcal V_1 \; : \; |t| < 1, \;\; {1 / 2} < \|z\| < 2\Big\}, \quad b(t,z)= (1 - \|z\| + it)\cdot z.
$$
Then, $\sdeg(b,\Omega_0)=  1 \cdot (\mathbb Z_1)$.

\medskip
\noi
\textbf{(A5)(Suspension)} Suppose that $A$ is an orthogonal $S^1$-representation and $U$ is an open bounded invariant neighborhood of zero in $A$.
Then, $$\sdeg(f \times {\rm Id}, \Omega \times U) = \sdeg(f,\Omega).$$

Using the equivariant version of the standard Leray-Schauder projection, one can define the $S^1$-degree to $S^1$-equivariant compact vector fields (see
\cite{AED, IV-B} for details). Combining the axioms of the $S^1$-degree with some standard homotopy theory techniques, one can reduce the computation of the $S^1$-degree of the maps naturally associated with the system undergoing the Hopf bifurcation to the computation of the Brouwer degree. To be more precise,  let $V$ be an orthogonal
$S^1$-representation with $V^{S^1}=\{v\in V:\ (\gamma, v)= v \ \ \forall \gamma \in S^1\}=\{0\}$.  Take the isotypical decomposition
$$
V = V_{k_1} \oplus V_{k_2} \cdots \oplus V_{k_s},$$
where each $V_{k_j}$ is modeled by the $k_j$-th irreducible representation. Define
$$\cO=\{(\lambda,v) \in \bbC \oplus V:\|v\|<2\ , \ {1}/{2}<|\lambda|<4\}.$$
Now, consider a map $a:S^1 \to GL^{S^1}(V)$ and define $a_j:S^1 \to GL^{S^1}(V_{k_j})$ by the formula $a_j(\lambda)= a(\lambda)_{|V_{k_j}}$ (see, \cite[p.~284]{AED}). Let $f_a:\overline\cO\to \bbR\oplus V$ be an
$S^1$-equivariant map defined by
$$f_a(\lambda,v)= \left(|\lambda|(\|v\|-1)+\|v\|+1, a\left(\frac{\lambda}{|\lambda|}\right)v \right).$$
The following formula plays an important role in the proof of Theorem \ref{thm:main-theorem}:
\begin{equation}\label{thm:main-computational-Theorem}
\s1deg(f_a,\cO)= \sum_{j=1}^s \left( \deg(\text{det}_{\mathbb C}\circ a_j,B)\right)(\bbZ_{k_j}),
\end{equation}
where $B$ stands for the unit ball in $\mathbb C$ (cf. \cite{AED}, Theorem 4.23).
















\subsection*{Acknowledgements}
The authors were supported by National Science Foundation grant DMS-1413223.
WK was also supported by Chutian Scholar Program at China Three Gorges University, Yichang,
Hubei (China).

\end{document}